\documentclass[11pt]{amsart}
\pdfoutput=1
\usepackage{amsmath, amssymb}
\usepackage{amsfonts}
\usepackage{mathrsfs}
\usepackage[arrow,matrix,curve,cmtip,ps]{xy}
\usepackage{graphicx}
\usepackage{float}
\usepackage{amsthm}
\usepackage{lpic}
\usepackage{caption}
\usepackage{hyperref}
\hypersetup{pdfauthor={Tarik Aougab, Jonah Gaster},pdftitle={lifting curves simply},colorlinks=true,linkcolor=black,citecolor=black}

\usepackage{nicefrac}
\usepackage{enumerate}
\usepackage{esvect}
\usepackage{framed}

\allowdisplaybreaks

\newtheorem{theorem}{Theorem}[section]
\newtheorem{innercustomthm}{Theorem}

\newtheorem{lemma}[theorem]{Lemma}
\newtheorem{proposition}[theorem]{Proposition}
\newtheorem*{proposition*}{Proposition}

\newtheorem{corollary}[theorem]{Corollary}
\newtheorem*{corollary*}{Corollary}

\newtheorem*{theorem*}{Theorem}
\theoremstyle{remark}
\newtheorem{remark}[theorem]{Remark}
\newtheorem{definition}[theorem]{Definition}

\newtheorem{question}{Question}

\newenvironment{customthm}[1]
	{\renewcommand\theinnercustomthm{#1}\innercustomthm}
	{\endinnercustomthm}

\numberwithin{equation}{section}

\newcommand{\Z}{\mathbb{Z}}

\newcommand{\R}{\mathbb{R}}

\begin{document}
\title[Complete 1-systems and dual cube complexes]{ Curves intersecting exactly once and their dual cube complexes}

\author{Tarik Aougab, Jonah Gaster}

\date{January 31, 2015}


\begin{abstract}

Let $S_{g}$ denote the closed orientable surface of genus $g$. We construct exponentially many mapping class group orbits of collections of $2g+1$ simple closed curves on $S_{g}$ which pairwise intersect exactly once, extending a result of the first author \cite{aougab} and further answering a question of Malestein-Rivin-Theran \cite{m-r-t}. To distinguish such collections up to the action of the mapping class group, we analyze their dual cube complexes in the sense of Sageev \cite{sageev}. In particular, we show that for any even $k$ between $\lfloor g/2 \rfloor$ and $g$, there exists such collections whose dual cube complexes have dimension $k$, and we prove a simplifying structural theorem for any cube complex dual to a collection of curves on a surface pairwise intersecting at most once. 

\end{abstract}

\maketitle

\section{Introduction}
\label{intro}

Let $S=S_{g}$ denote the closed orientable surface of genus $g$, let $\mbox{Mod}^*(S)$ denote the associated \textit{extended mapping class group}, and let $\mathcal{I}(S)$ be the set of isotopy classes of essential simple closed curves on $S$. There is a natural action of $\mbox{Mod}^*(S)$ on $\mathcal{I}(S)$, and analyzing orbits of various finite subsets of $\mathcal{I}(S)$ has proven to be a fruitful way of probing the algebra and geometry of $\mbox{Mod}^*(S)$. For example, the $\mbox{Mod}^*(S)$-orbits of pants decompositions have been used to estimate the Weil-Petersson diameter of the thick part of Moduli space \cite{cavendish-parlier}. 

The main focus of this paper is to explicitly construct many distinct $\mbox{Mod}^*(S)$-orbits of collections of curves with intersection properties that are reminiscent of pants decompositions. We show:

\begin{customthm}{1}
\label{MainThm} For $g\ge 2$, there exist at least $2^{g-3}/(g-1)$ and at most $(4g^2+2g)!$ distinct $\mbox{Mod}^*(S)$-orbits of collections of $2g+1$ simple closed curves pairwise intersecting once. 
\end{customthm}

Malestein-Rivin-Theran have shown that any collection of curves pairwise intersecting once has cardinality at most $2g+1$. Thus we think of such collections as analogous to pants decompositions in the following way: pants decompositions are the maximal cliques of the \textit{curve graph}, $\mathcal{C}(S)$, of $S$ $-$ the graph with vertex set $\mathcal{I}(S)$, and edges between two isotopy classes that can be realized disjointly on $S$. Similarly, the curve systems in Theorem \ref{MainThm} are the largest cliques of the \textit{Schaller} or \textit{systole graph}, $\mathcal{SC}(S)$, of $S$, whose vertices correspond to the subset of $\mathcal{I}$ of non-separating simple closed curves, and whose edges correspond to pairs of curves intersecting exactly once. Note that, unlike pants decompositions, there do exist non-maximum collections of curves pairwise intersecting once that are nonetheless maximal with respect to inclusion.

We remark that the number of $\mbox{Mod}^*(S)$-inequivalent pants decompositions grows at least factorially in $g$ (see \cite{bollobas} for an asymptotically precise count), and we conjecture that the same is true for the types of curve systems considered here. However, though $\mathcal{SC}(S)$ and $\mathcal{C}(S)$ are $\mbox{Mod}^*(S)$-equivariantly quasi-isometric, it is not necessarily the case that the corresponding growth rates of $\mbox{Mod}^*(S)$-inequivalent maximal cliques are comparable. This question requires a more detailed understanding of the specific nature of $\mathcal{SC}(S)$. 

Schaller has shown that the automorphism group of $\mathcal{SC}(S)$ is the (extended) mapping class group \cite{schmutz-schaller}, and thus as a corollary to Theorem \ref{MainThm} we obtain:

\begin{corollary*} \label{Clique} The number of maximum cardinality cliques of $\mathcal{SC}(S)$, inequivalent under the action of $\mbox{Aut}(\mathcal{SC}(S))$, grows at least exponentially in $g$.  
 \end{corollary*}

By a $k$-system, we mean any subset $\Gamma \subset \mathcal{I}(S)$ consisting of curves pairwise intersecting at most $k$ times. A \textit{complete} $k$-system is a $k$-system in which any two curves intersect \textit{exactly} $k$ times. Thus the main focus of this paper is the study of complete $1$-systems of maximum possible size, or \textit{maximum complete 1-systems}. 

Malestein-Rivin-Theran showed that such $1$-systems are unique up to the action of $\mbox{Mod}^*(S)$ for $g=1,2$, and they asked if this uniqueness persists for higher genera. The first author answered this question by subsequently constructing two distinct orbits of complete $1$-systems of size $2g+1$ on $S$, for all $g \geq 3$. Thus we view Theorem \ref{MainThm} as a further demonstration of the non-uniqueness of maximum complete $1$-systems. 

Note that `complete'-ness, for the 1-systems we consider, is a significant simplifying assumption. Though there has been substantial recent progress towards estimating the size of maximum 1-systems \cite{przytycki}, even asymptotically precise counts are not currently available. While it would be interesting to examine the number of $\mbox{Mod}^*(S)$-orbits of maximum 1-systems, the absence of any examples when $g\ge 3$ makes this seem difficult\footnote{\cite{m-r-t} calculate that there are two $\mbox{Mod}^*(S)$-orbits of maximum 1-systems for $g=2$.}.

Our method of distinguishing $\mbox{Mod}^*(S)$-orbits of a curve system $\Gamma$ is to analyze the \textit{dual cube complex} $C(\Gamma)$ to $\Gamma$, a complex built from cubes of various dimensions which encodes the combinatorics of the intersections between curves. This invariant is a useful way of organizing topological information about the $\mbox{Mod}^*(S)$-orbit of $\Gamma$. Along the way in our analysis, we show:

\begin{customthm}{2} \label{CubeChar} Let $\Lambda_{1}, \Lambda_{2}$ be any two collections of curves which fill a closed surface $S$. Then $\Lambda_{1}$ and $\Lambda_{2}$ are equivalent under the action of the extended mapping class group if and only if there is an isomorphism of cube complexes $C(\Lambda_{1}) \cong C(\Lambda_{2})$. The induced set map from $\Lambda_1$ to $\Lambda_2$ corresponds to the induced map between hyperplanes of $C({\Lambda_1})$ and hyperplanes of $C({\Lambda_2}$).
\end{customthm}

Thus, the reader may view the main result as a construction of many non-isomorphic cube complexes, each dual to a maximum complete $1$-system. In particular, we show:

\begin{proposition*} \label{CubeReal} For any even $k \in [\lfloor g/2\rfloor,g]$, there exists a complete $1$-system of size $2g+1$ on $S_{g}$ whose dual cube complex has dimension $k$. 
\end{proposition*}

It is interesting to consider whether or not the dimension of the cube complex dual to a maximum complete 1-system grows with the genus. At the moment this problem seems difficult. As a first step, one might determine whether there exists a maximum complete $1$-system $\Gamma$ whose dual cube complex is $2$-dimensional:

\begin{question} \label{2d} Does there exist a maximum complete $1$-system whose dual cube complex is $2$-dimensional?
\end{question}

In this case, the quotient of the dual cube complex $C(\Gamma)$ by the action of $\pi_{1}(S)$ produces a square-tiled copy of $S$, with at least $4$ squares around each vertex. We conjecture that the answer to Question \ref{2d} is no. 

In general, $C(\Gamma)$ can be a very complicated combinatorial object; indeed, one may interpret this as a consequence of Theorem \ref{CubeChar}, since $\mbox{Mod}^*(S)$-orbits of curve systems can be difficult to distinguish (cf.~\cite{levitt-vogtmann}). In order to leverage $C(\Gamma)$ to useful information about $\mbox{Mod}^*(S)$, we make use of the following simplifying theorem for cube complexes dual to $1$-systems:

\begin{customthm}{3} \label{3-to-n} Suppose $\Gamma= \left\{\gamma_{1},...,\gamma_{n}\right\}$ is any $1$-system on an orientable surface $S$, possibly with boundary. Then the dimension of $C(\Gamma)$ is $n$ if and only if the dimension of $C(\Gamma')$ is $3$, for $\Gamma'$ any triple of curves in $\Gamma$. 
\end{customthm}

\begin{remark} We note that one direction of Theorem \ref{3-to-n} is immediate: if the dimension of the entire cube complex is $n$, then any three curves must correspond to a $3$-cube in the dual cube complex. However, as Figure \ref{necessity1system} in Section \ref{qualitative info section} demonstrates, the converse is false if the assumption of being a $1$-system is dropped.  
\end{remark}

Our main construction requires $g$ to be odd, and we extend the conclusion of Theorem \ref{MainThm} and Proposition \ref{CubeReal} to even $g$ via the following simple process (see \S\ref{stabilizing} for a slightly more careful description):

Beginning with a complete $1$-system $\Gamma$ of size $2g+1$ on $S_{g}$ for $g= 2k+1$, excise a pair of small open disks which are locally on opposite sides of some $\gamma \in \Gamma$, and glue on an annulus $A$ along the resulting boundary circles. Note that $\Gamma$ is still a complete $1$-system on $S_{g+1}$, and we extend $\Gamma$ to a collection of $2g+3$ curves by adding $\gamma', \gamma''$, defined as follows: both $\gamma', \gamma''$ run parallel to $\gamma$ in the complement of the new annulus $A$. Within $A$, both $\gamma', \gamma''$ run from one boundary component of $A$ to the other, intersecting once in the interior of $A$. 

Therefore $\gamma', \gamma''$ intersect each other exactly once, within $A$, and each intersects all of the original elements of $\Gamma$ exactly once because $\gamma$ does. If $\Gamma$ on $S_{g}$ is obtained from a complete $1$-system $\Gamma'$ on $S_{g-1}$ as described above, we call $\Gamma$ a \textit{stabilization} of $\Gamma'$. 

It is natural to ask whether or not \textit{every} complete $1$-system on $S_{g}$ of size $2g+1$ is obtained from one on $S_{g-1}$ of size $2g-1$ via this process:

\begin{question} \label{Stab} Let $\Gamma$ be a complete $1$-system on $S_{g}$ of size $2g+1$. Is it always the case that $\Gamma$ is a stabilization of some complete $1$-system on $S_{g-1}$?
\end{question}

We observe that each complete $1$-system we construct is indeed a stabilization of a complete $1$-system on a lower genus surface. Furthermore, we note that a positive answer to Question \ref{Stab} implies a negative answer to Question \ref{2d}: Lemmas \ref{tool1} and \ref{stabilized 3-cubes general 1} imply that the dimension of the cube complex dual to any complete $1$-system obtained via stabilization is at least three.

\textbf{Organization of paper.} In \S\ref{background} and \S\ref{cube cplx construction section} we outline some preliminary notions regarding the mapping class group and Sageev's construction of dual cube complexes. In \S\ref{cube complex characterizes section}, we prove Theorem \ref{CubeChar}; in \S\ref{qualitative info section}, we prove Theorem \ref{3-to-n}; in \S\ref{construction section}, we outline the main construction of our complete $1$-systems; in \S\ref{using dual cube complex section}, \S\ref{polygon section}, and \S\ref{stabilizing} we prove that these complete $1$-systems are indeed inequivalent up to the action of the extended mapping class group, completing the proof of Theorem \ref{MainThm}.

\textbf{Acknowledgements.} Both authors thank David Dumas for numerous helpful conversations and suggestions. The second author thanks as well Peter Shalen for pointing out the relevance of Sageev's invariant, and Marc Culler and Daniel Groves for their patience and time.

\section{background}
\label{background}
Let $\Gamma=\{\gamma_{1},\ldots,\gamma_{n}\}$ be a collection of free homotopy classes of closed curves on $S$. Recall that a collection of curves form a \emph{bigon} if there is an embedded disk in $S$ whose boundary is the union of two arcs of the curves. A \emph{minimal position realization} of $\Gamma$ is a set $\lambda=\{\eta_{1},\ldots,\eta_{n}\}$ such that:
\begin{enumerate}[(i)]
\item Each $\eta_{i}:S^{1}\to S$ is a smooth immersion in the free homotopy class $\gamma_{i}$.
\item The union $\bigcup_{i}\eta_{i}(S^{1})$ forms no bigons.
\item The immersed submanifolds $\{\eta_{1}(S^{1}),\ldots,\eta_{n}(S^{1})\}$ intersect only at transverse double points.
\end{enumerate}
We will refer to minimal position realizations simply as \emph{realizations}. In everything that follows, we suppose that $\lambda=\{\eta_{1},\ldots,\eta_{n}\}$ is a realization of $\Gamma$. Condition (ii) above implies that $\lambda$ minimizes the sum of the pairwise geometric intersection numbers of the curves in $\Gamma$ \cite[Prop.~1.7, p.~31]{farb-margalit}. See \cite[Ch.~1]{farb-margalit} for background on curves on surfaces.

Let $\Lambda$ indicate the set of the lifts of elements of $\lambda$ to the universal cover $\widetilde{S}$, so that the elements of $\Lambda$ are curves in $\widetilde{S}$. The union of the curves in $\Lambda$ may be considered as an embedded graph $G\subset\widetilde{S}$. Condition (iii) above guarantees that each vertex of this graph has valence four, and condition (ii) implies that every pair of curves in $\Lambda$ intersect at most once (see \cite[Lemma 1.8, p.~30]{farb-margalit}).  

When the curves in $\Gamma$ are disjoint, then the dual graph to the lifts $\Lambda$ admits an isometric action of $\pi_1 S$, and the quotient graph is an invariant for $\Gamma$ that can be used to distinguish mapping class group orbits. However, in general the dual graph to $\lambda$ in $S$ is not an invariant of $\Lambda$, as different realizations may yield non-isomorphic graphs. For example, the presence of a triangle in the complement of $\lambda$ allows a Reidemeister type III move, creating a new realization but changing the isomorphism types of the dual graph, as shown in Figure \ref{reidemeister}.

\begin{figure}[h]
	\begin{minipage}[]{.45\linewidth}
		\centering
		\vspace{.2cm}
		\includegraphics[width=5cm]{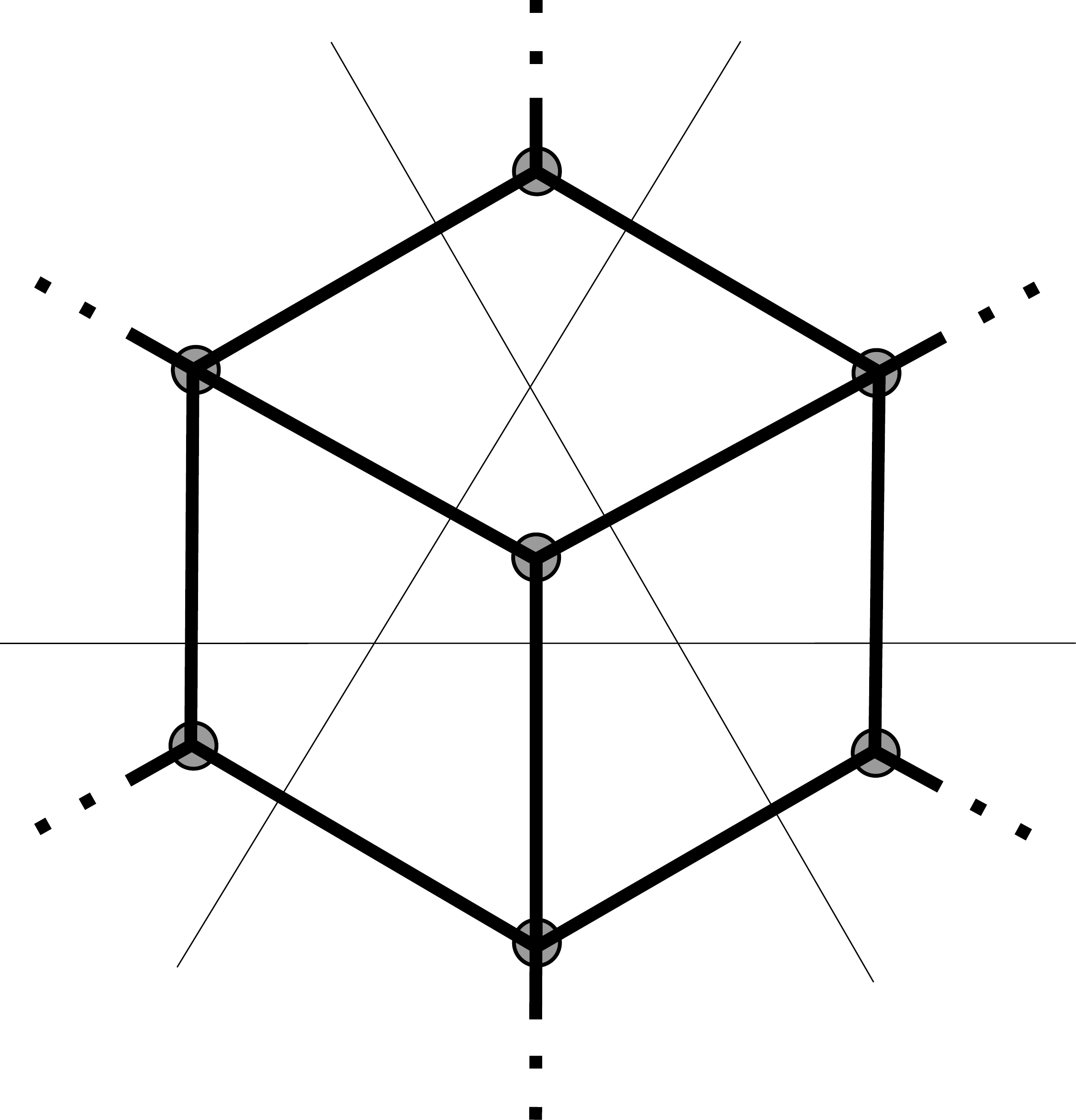}
		\vspace{.4cm}
	\end{minipage}
	\hspace{.5cm}
	\begin{minipage}[]{.45\linewidth}
		\centering
		\vspace{.2cm}
		\includegraphics[width=5cm]{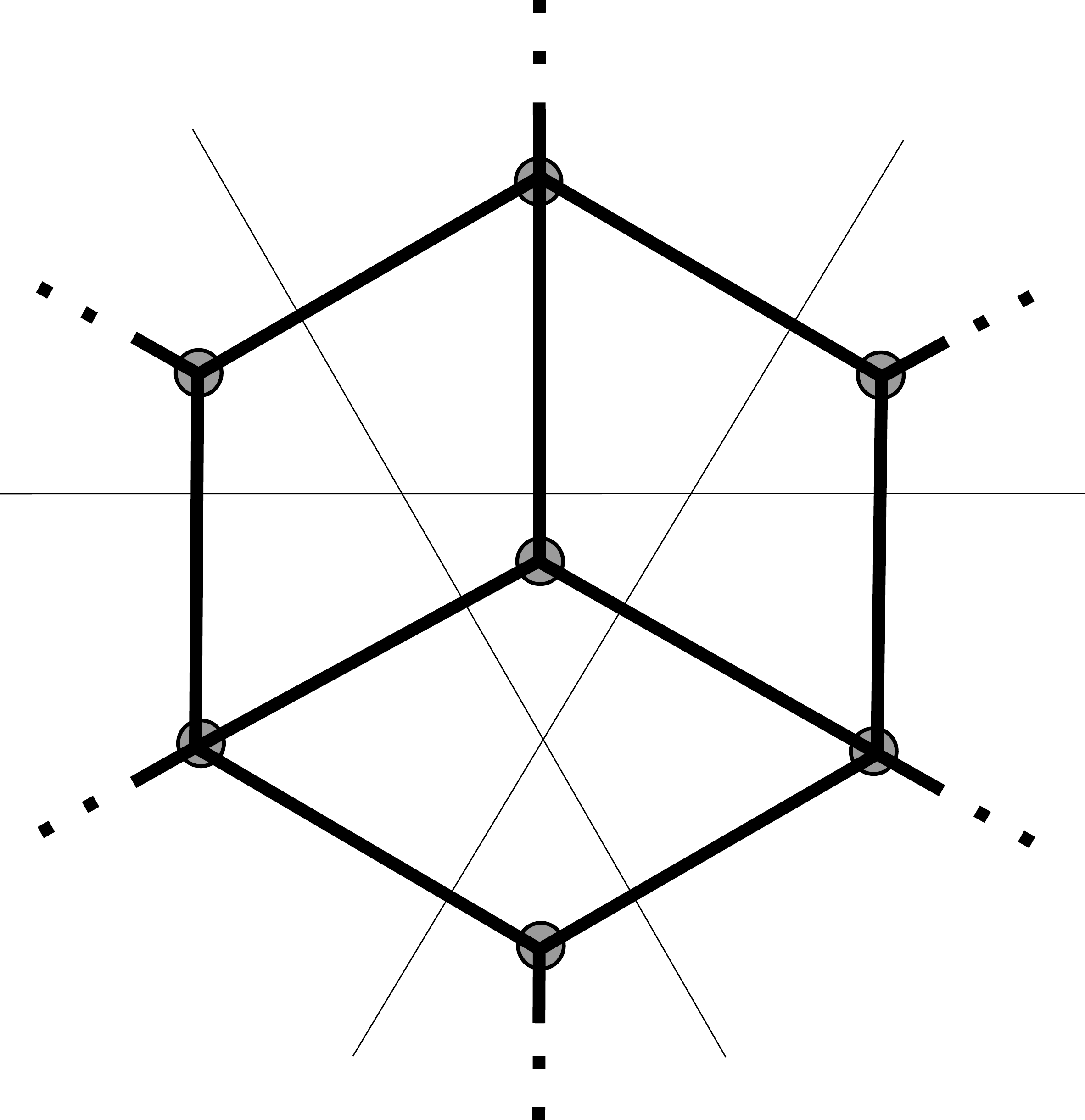}
		\vspace{.4cm}
	\end{minipage}
\caption{Changing the realization $\lambda$ by a Reidemeister type III move changes the isomorphism type of the dual graph.}
\label{reidemeister}
\end{figure}

While the dual graph depends essentially on the realization $\lambda$, we now describe the construction of a related cube complex `dual' to $\Lambda$ which is independent of realization. Originally due to Sageev, this produces an isometric action of $\pi_{1}S$ on a finite-dimensional cube complex, which we denote by $\widetilde{C(\lambda)}$, with quotient $C(\Lambda):=\widetilde{C(\lambda)}/\pi_{1}S$, a cube complex with finitely many maximal cubes. Though it will be unnecessary in this work, when $\Lambda$ is filling $C(\Lambda)$ is non-positively curved, and this construction can be placed in a considerably more general context (see \cite{sageev}, \cite{sageev2}, and \cite{chatterji-niblo}).

Recall that a \emph{cube} (of dimension $n$) is the cell complex $[-1,1]^{n}$, and a \emph{cube complex} is the quotient of disjoint cubes by a gluing map which is a Euclidean isometry on each cell. In such a complex, a \emph{maximal cube} is a cube which is not contained in a higher dimensional cube. The \emph{dimension} of a cube complex is the supremum of the dimensions of its cubes. A \emph{square complex} is a two-dimensional cube complex. See \cite{sageev2} for an introduction to cube complexes.

The \emph{local hyperplanes} of an $n$-cube are the intersections of $[-1,1]^{n}$ with coordinate planes. We introduce an equivalence relation on local hyperplanes in a cube complex: Two local hyperplanes are \emph{hyperplane equivalent} if they intersect along the $1$-skeleton of the cube complex, and a \emph{hyperplane} of a cube complex is the hyperplane equivalence class of a local hyperplane. An isomorphism of cube complexes is a homeomorphism that is a cellular Euclidean isometry, which thus preserves hyperplanes. We will denote the hyperplanes of a cube complex $C$ by $\mathcal{H}_{C}$. Note that $1$-dimensional cube complexes are graphs and trees, whose hyperplanes (and local hyperplanes) are midpoints of edges. 

\section{Sageev's construction}
\label{cube cplx construction section}

Fix a choice of realization $\lambda$ of $\Gamma$, with lifts $\Lambda$ as before. We now describe the construction of a cube complex $\widetilde{C(\lambda)}$ that is independent of the choice of realization. Each element $\gamma\in\Lambda$ separates $\widetilde{S}\setminus\gamma$ into two connected components. Fix a choice of identification of these two half-spaces with $\{1,-1\}$ for each $\gamma\in\Lambda$. We will refer to a choice of one of these two as a \emph{labeling} of $\gamma$, and a labeling of $\Lambda$ is a labeling for each of the curves of $\Lambda$. Identifying $2$ with $\{1,-1\}$, we say a labeling $v\in2^{\Lambda}$ is \emph{admissible} if the half-spaces $v(\alpha)$ and $v(\beta)$ intersect for every pair of curves $\alpha,\beta\in\Lambda$.

Let $\mathcal{V}_{0}\subset2^{\Lambda}$ denote the collection of admissible labelings. Note that for each connected region $U\subset\widetilde{S}\setminus\Lambda$, there is an admissible labeling $v_{\lambda}(U)\in\mathcal{V}_{0}$ defined as follows: For each curve in $\Lambda$, $v_{\lambda}(U)$ chooses the half-space containing $U$. Let $\mathcal{V}_1'$ be the graph whose vertex set is $\mathcal{V}_{0}$, where two labelings are joined by an edge when they differ on exactly one element of $\Lambda$. Hence there is an element of $\Lambda$ associated to each edge of $\mathcal{V}_1'$.  

Choose a connected region $U\subset\widetilde{S}\setminus\Lambda$, and let $\mathcal{V}_1$ denote the connected component of $v_{\lambda}(U)$ in $\mathcal{V}_1'$. (The choice of connected region $U$ is evidently not essential). When lifts $\gamma_1 , \ldots , \gamma_n \in \Lambda$ pairwise intersect, it is straightforward to check that $\mathcal{V}_1$ contains an embedded copy of the 1-skeleton of $[-1,1]^n$. We obtain the cube complex $\widetilde{C(\lambda)}$ by adding in the interior of any cube whose $1$-skeleton is contained in $\mathcal{V}_{1}$; the links of the resulting cube complex are all flag simplicial complexes (see \cite{chatterji-niblo} for a more detailed exposition). 

The action of $\pi_{1}S$ on $\Lambda$ induces a permutation action of $\pi_{1}S$ on $2^{\Lambda}$: For $g\in\pi_{1}S$ and admissible labeling $v\in\mathcal{V}_0$, the labeling $g\cdot v\in\mathcal{V}_0$ is given by $(g\cdot v)(\gamma)=g^{-1}\cdot v(g\cdot\gamma)$ for each $\gamma\in\Lambda$. Since the action on $\widetilde{S}$ is by deck transformations, it is straightforward to check that the action of $\pi_{1}S$ preserves the admissible set $\mathcal{V}_0$ and induces an action by graph automorphisms on the $1$-skeleton $\mathcal{V}_1$. This action extends to $\widetilde{C(\lambda)}$ by definition.  

The complex $\widetilde{C(\lambda)}$ with an action of $\pi_1 S$ is independent of the realization $\lambda$: a choice of half space for a lift $\gamma\in\Lambda$ determines, and is determined by, a choice of complement of $\partial \gamma \subset S^1=\partial \pi_1 S$. By [F-M, bigons], choices of half-spaces for two lifts $\gamma,\gamma'\in\Lambda$ intersect if and only if the endpoints of $\gamma$ link with those of $\gamma'$ on $S^1$. The latter is independent of the choice of realization. We thus denote $\widetilde{C(\lambda)}$ by $\widetilde{C(\Gamma)}$ when convenient.

Moreover, it is immediate that $\widetilde{C(\Gamma)}$ is non-positively curved. A direct argument shows that any loop in $\widetilde{C(\Gamma)}$ must backtrack, which implies that the complex is also simply-connected, and thus CAT(0). We collect the relevant information about $\widetilde{C(\Gamma)}$:

\begin{theorem}[Sageev]
\label{sageev's}
{Suppose that $\Gamma$ is a filling collection of curves. Then the cube complex $\widetilde{C(\Gamma)}$ is $CAT(0)$, and the action of $\pi_{1}S$ is free, properly discontinuous, and cocompact. Given realization $\lambda$, there is a $\pi_{1}S$-equivariant incidence-preserving identification of $\mathcal{H}_{\widetilde{C(\Gamma)}}$ with the lifts $\Lambda$.}
\end{theorem} 

In what follows, we do not return to the combinatorial definition of $C(\Gamma)$.  We will invoke the correspondence in Theorem \ref{sageev's} often, which we may briefly refer to as the \emph{curves to hyperplanes correspondence}.

\section{Mapping class group orbits of collections of curves}
\label{cube complex characterizes section}
We now characterize a filling curve system from the isomorphism type of its dual cube complex. Recall the \emph{extended mapping class group} $\mbox{Mod}^{*}(S)=\mathrm{Diff}(S)/\mathrm{Diff}_0(S)$. We recall Theorem \ref{CubeChar}:

\begin{customthm}{2}
{Two filling curve systems $\Lambda_1$ and $\Lambda_2$ are equivalent under the action of $\mbox{Mod}^{*}(S)$ if and only if there is an isomorphism of cube complexes $C({\Lambda_1})\cong C({\Lambda_2})$. The induced set map from $\Lambda_1$ to $\Lambda_2$ corresponds to the induced map between hyperplanes of $C({\Lambda_1})$ and hyperplanes of $C({\Lambda_2}$).}
\end{customthm}

\begin{proof}
{One direction is straightforward: Suppose $\phi\cdot\Lambda_1=\Lambda_2$ for $\phi\in \mbox{Mod}^{*}(S)$. Choose a realization $\lambda_{1}$ for $\Lambda_{1}$, and a homeomorphism $\phi'$ realizing $\phi$. In this case $\phi'$ induces a $\pi_{1}S$-equivariant isomorphism $\widetilde{\phi'}:\widetilde{C(\lambda_1)}\cong\widetilde{C(\phi\cdot\lambda_1)}$, where the induced map of hyperplanes corresponds to the set map induced by $\phi'$ from $\lambda_1$ to $\phi'\cdot\lambda_1$. Since $\phi'\cdot\lambda_{1}$ is a realization of $\Lambda_{2}$, we have $C(\Lambda_{1})\cong C(\lambda_{1})$ and $C(\Lambda_{2})\cong C(\phi\cdot\lambda_{1})$. The result follows.

On the other hand, suppose $\Phi:C(\Lambda_1)\cong C(\Lambda_2)$ is an isomorphism of cube complexes. Choose realizations $\lambda_{1}$ and $\lambda_{2}$ for $\Lambda_{1}$ and $\Lambda_{2}$. By Theorem \ref{sageev's}, we have that $\widetilde{C(\Lambda_{i})}$ is simply-connected and the action of $\pi_{1}S$ is free and properly discontinuous, for each $i=1,2$. Thus $\widetilde{C(\Lambda_{i})}$ is the universal cover of $C(\Lambda_{i})$ and there are isomorphisms $\pi_{1}S\cong\pi_{1}C(\Lambda_{i})$ for each $i$. Composing these isomorphisms with $\Phi_{*}$, we find an automorphism $\phi_{*}:\pi_{1}S\to\pi_{1}S$. By construction, we may lift $\Phi$ to a $\phi_{*}$-equivariant isomorphism of cube complexes $\widetilde{\Phi}:\widetilde{C(\Lambda_{1})}\to\widetilde{C(\Lambda_{2})}$, in the sense that $$\widetilde{\Phi}(g\cdot x)=\phi_{*}(g)\cdot\widetilde{\Phi}(x)$$
for each $x\in\widetilde{C(\Lambda_{1})}$ and $g\in\pi_{1}S$. This induces a corresponding equivariant map on hyperplanes, which in turn induces a correspondence of the collections of conjugacy classes of stabilizers of hyperplanes of $\widetilde{C(\Lambda_{1})}$ with those of $\widetilde{C(\Lambda_{2})}$. 

The curves to hyperplanes correspondence guarantees that there is an identification of the collection of conjugacy classes of hyperplane stabilizers of $\pi_{1}S$ acting on $\widetilde{C(\Lambda_{i})}$ with the collection of conjugacy classes of stabilizers of curves in $\widetilde{\lambda_{i}}$, and so we arrive at a correspondence of the collection of conjugacy classes of stabilizers of curves in $\widetilde{\lambda_{1}}$ with those of $\widetilde{\lambda_{2}}$. The conjugacy classes of the stabilizers of curves in $\widetilde{\lambda_{i}}$ are naturally identified with the collection of conjugacy classes of $\Lambda_{i}$ in $\pi_{1}S$, from which it follows that the automorphism $\phi_{*}$ of $\pi_{1}S$ takes the conjugacy classes determined by $\Lambda_{1}$ to those of $\Lambda_{2}$.  By the Dehn-Nielsen-Baer Theorem \cite[Ch.~8, Thm.~8.1]{farb-margalit}, there is $\phi\in \mbox{Mod}^{*}(S)$ inducing $\phi_{*}$, so that $\phi\cdot\Lambda_{1}=\Lambda_{2}$.}
\end{proof} 

\section{Recognizing $n$-cubes dual to a 1-system}
\label{qualitative info section}
There remains the problem of recognizing quantitative information about $C(\Lambda)$ from a given set of curves $\Lambda$. In this section, we prove Theorem \ref{3-to-n}, giving a criterion for recognizing the dimensions of cubes in the complex dual to a 1-system.

A realization of a collection of curves forms a \emph{triangle} if there is an embedded disk on $S$ whose boundary components are three arcs of the curves, and so that these arcs intersect pairwise exactly once on the boundary of the disk. A collection of homotopy classes of curves forms a triangle if there is a realization of the curves that forms a triangle.

\begin{lemma}
\label{triangles}
{If a collection of homotopy classes of closed curves forms a triangle then every realization of the curves forms a triangle.}
\end{lemma}

\begin{proof}
{By \cite[Thm.~1]{graaf-schrijver} any two realizations are homotopic through isotopies and finitely many Reidemeister type III moves. Neither of these changes the existence of a triangle in the complement.}
\end{proof}

\begin{lemma}
\label{tool1}
{If $\Gamma=\{\gamma_{1},\gamma_{2},\gamma_{3}\}$ is a realization of simple closed curves on $S$, then $\dim C(\Gamma)=3$ if and only if $S\setminus\Gamma$ has a connected component that is a triangle.}
\end{lemma}

\begin{proof}
{If $\dim C(\Gamma)=3$ and $\Gamma$ is in minimal position, then the curves to hyperplanes correspondence guarantees that there are three mutually intersecting lifts of $\Gamma$ in $\widetilde{S}$. As $\Gamma$ is a realization, the intersections of these lifts are not concurrent, and there is a triangle $T$ in their complement. Since $T$ is compact, the intersection $T\cap\widetilde{\Gamma}$ consists of finitely many arcs. As each such arc doesn't form a bigon with the boundary of the triangle, the complement of the arc has a triangular component. Thus there is a triangular component $T'\subset T\setminus\widetilde{\Gamma}$ (cf. with \emph{innermost} bigon, \cite[p.~31]{farb-margalit}).

Let $\pi:\widetilde{S}\to S$ be the covering map. If $int(T')$ does not embed under $\pi$, then there is an element $g\in\pi_{1}S$ so that $int(T')\cap g\cdot int(T')\ne\emptyset$. In this case, by the Jordan Curve Theorem there is an intersection $p\in\partial T'\cap g\cdot\partial T'$. Since $\partial T'$ consists of arcs from lifts of $\Gamma$, this violates $T'\subset T\setminus\widetilde{\Gamma}$. Thus $int(T')$ embeds in $S$, and we have found a triangle in the complement of $\Gamma$.

Conversely, suppose $\Gamma$ forms a triangle $T$. Lift this topological disk to $\widetilde{S}$, and observe that the arcs of the boundary are contained in curves that pairwise intersect. By the curves to hyperplanes correspondence, there is a $3$-cube corresponding to this collection. Using the correspondence again, a $4$-cube would yield four lifts of curves in $\Gamma$ that pairwise intersect. Since $|\Gamma|=3$, two of these lifts are in the same $\pi_{1}S$-orbit. As each of the $\gamma_{i}$ are simple, this is impossible.}
\end{proof}

Recall that a \emph{1-system} is a collection of homotopy classes of simple closed curves whose pairwise geometric intersection number is at most one. It is interesting to note that our proof of Theorem \ref{3-to-n} below makes essential use of the orientation of $S$. 

\begin{customthm}{3}
{Suppose $\Gamma=\{\gamma_{1},\ldots,\gamma_{n}\}$ is a $1$-system on $S$. Then the dimension of $C(\Gamma)$ is $n$ if and only if the dimension of $C(\Gamma')$ is three, for every triple $\Gamma'\subset\Gamma$.}
\end{customthm}

One direction is straightforward: If $\dim C(\Gamma)=n$, then the correspondence of curves to hyperplanes guarantees that there is a set of $n$ lifts of the curves of $\Gamma$ which mutually intersect. In this case, every trio of these lifts mutually intersect and form a $3$-cube in the complex dual to that trio. Towards the other direction, using Lemma \ref{tool1} we may assume that each trio forms a triangle.

We define an \emph{almost-realization} of a curve system to be a minimal position realization (see \S\ref{background}) with one slight change: the words `only at transverse double points' in condition (iii) should be replaced by `transversally'. (The terms \emph{pairwise minimal position realization} would be more descriptive, but less economical). Note that \cite[Prop.~1.7, p.~31]{farb-margalit} still guarantees that almost-realizations minimize pairwise geometric intersection numbers, as with realizations.

\begin{proposition}
\label{tool2claim}
{If $\{\gamma_{1},\ldots,\gamma_{n}\}$ is a $1$-system so that every trio $\{\gamma_{i},\gamma_{j},\gamma_{k}\}$ forms a triangle, then there is an almost-realization $\{\alpha_{1},\ldots,\alpha_{n}\}$ of $\{\gamma_{1},\ldots,\gamma_{n}\}$ so that the curves $\alpha_{1},\ldots,\alpha_{n}$ have a single common intersection point.}
\end{proposition}

\begin{figure}
	\centering
	\includegraphics[height=6cm]{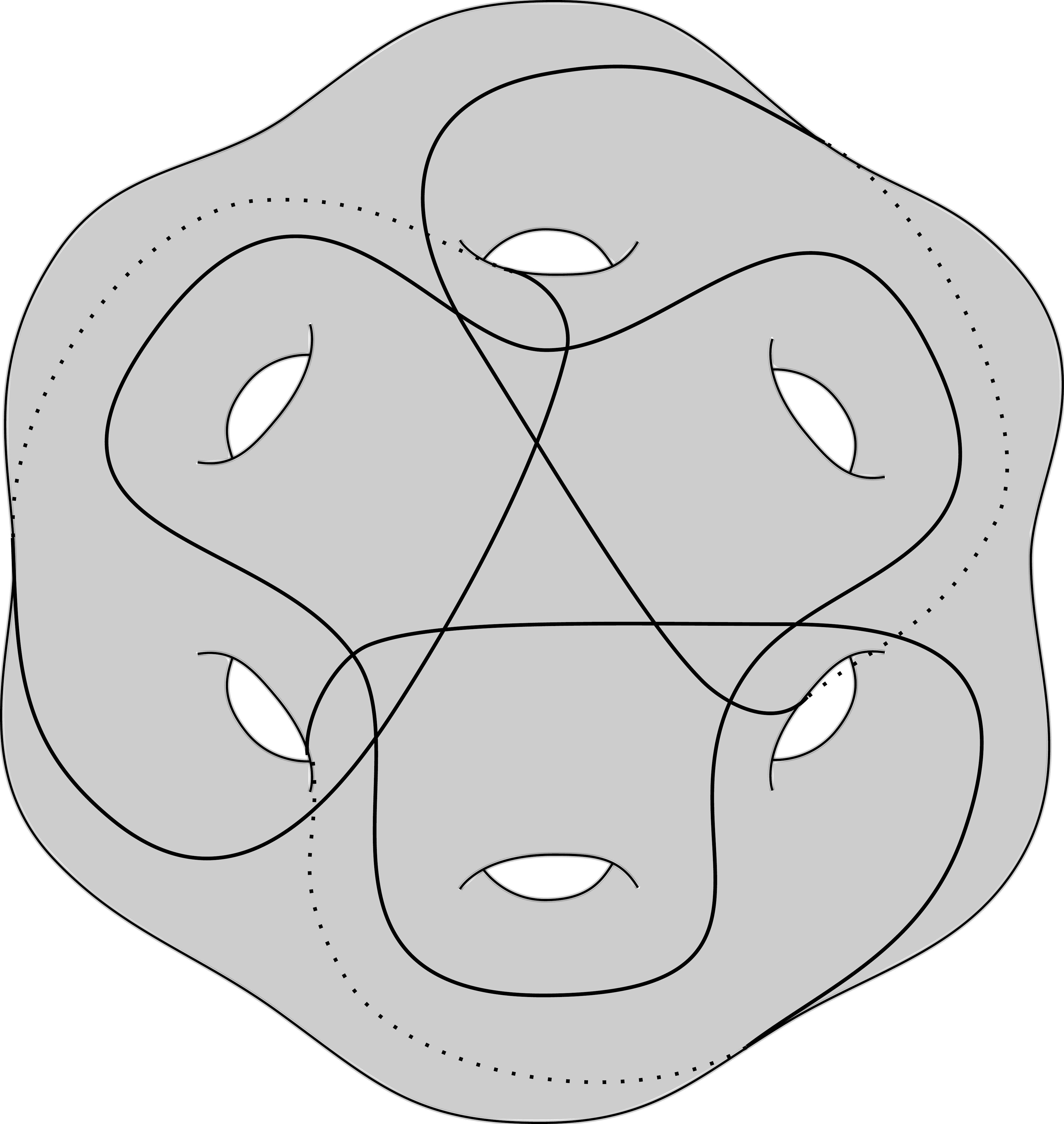}
	\caption{Four curves for which every trio forms a triangle, but so that there is no almost-realization with a single fourfold intersection point.}
	\label{necessity1system}	
\end{figure}

Note that it is essential that the curves form a $1$-system, as Figure \ref{necessity1system} exhibits four curves so that every trio forms a triangle, but the conclusion of Proposition \ref{tool2claim} fails. Note as well that, if the curves form a $1$-system and every trio forms triangles, then each pair intersects exactly once, which we assume below. We first prove Theorem \ref{3-to-n}, assuming Proposition \ref{tool2claim}.

\begin{proof}[Proof of Theorem \ref{3-to-n}]
{By Proposition \ref{tool2claim}, the curves $\gamma_{1},\ldots,\gamma_{n}$ have an almost-realization as $\alpha_{1},\ldots,\alpha_{n}$, so that the $\alpha_{i}$ have a single common intersection point. Choose a lift of this point to $\widetilde{S}$, and consider the lifts of the $\alpha_{i}$ passing through this point. Since the $\alpha_{i}$ form an almost-realization, there are no bigons formed by the curves, and the chosen lifts pairwise intersect in exactly one point. This implies that in any realization of the curves $\gamma_{1},\ldots,\gamma_{n}$ these lifts will pairwise intersect. By the curves to hyperplanes correspondence, there is an $n$-cube dual to these curves.}
\end{proof}

In the proof of Proposition \ref{tool2claim}, we will need a straightforward adaptation of \cite[Prop.~1.7, p.~31]{farb-margalit}:

\begin{lemma}
\label{homotopies to only one}
{If $\{\alpha_{1},\ldots,\alpha_{n}\}$ is an almost-realization of a collection of simple closed curves, then for any other simple closed curve $\beta$, there is a curve $\beta'$, homotopic to $\beta$, so that $\{\alpha_{1},\ldots,\alpha_{n},\beta'\}$ is an almost-realization.}
\end{lemma}

\begin{proof}
{By \cite[Prop.~1.7, p.~31]{farb-margalit}, two curves are in minimal position if and only if they form no bigons. One half of this proof shows that if a bigon is formed \cite[p.~32]{farb-margalit}, one may homotope one of these curves across the bigon while the other remains fixed. Choose any realization $\beta_{0}$ of $\beta$. If $\beta_{0}$ is in minimal position with $\alpha_{1}$, then let $\beta_{1}=\beta_{0}$. If $\beta_{0}$ is not in minimal position with $\alpha_{1}$, find a bigon that $\beta_{0}$ forms with $\alpha_{1}$. Use this bigon to find a curve $\beta_{0}'$, homotopic to $\beta_{0}$, so that $|\beta_{0}'\cap\alpha_{1}|<|\beta_{0}\cap\alpha_{1}|$.

It is crucial to observe that this step may be done so that $|\beta_{0}'\cap\alpha_{j}|\le|\beta_{0}\cap\alpha_{j}|$, for each $j$: The new curve $\beta_{0}'=\beta_{+}\cup\beta_{-}$ is composed of two arcs, where $\beta_{+}$ follows $\beta_{0}$ and $\beta_{-}$ follows $\alpha_{1}$. If an arc of $\alpha_{j}$ intersects $\beta_{0}'$ in the arc $\beta_{+}$, then there is a corresponding point of intersection of $\alpha_{j}$ with $\beta_{0}$. If an arc of $\alpha_{j}$ intersects $\beta_{0}'$ in the arc $\beta_{-}$, then this arc of $\alpha_{j}$ enters the bigon formed by $\beta_{0}$ and $\alpha_{1}$ along the side contained in $\alpha_{1}$. In this case, $\alpha_{j}$ must exit the bigon through the side contained in $\beta_{0}$, since $\alpha_{1}$ and $\alpha_{j}$ form no bigons by hypothesis (see \ref{bigon1} and \ref{bigon2}). Thus there is again a point of intersection of $\beta_{0}$ with this arc of $\alpha_{j}$ that corresponds to the intersection of $\alpha_{j}$ with $\beta_{0}'$. We conclude that $|\beta_{0}'\cap\alpha_{j}|\le|\beta_{0}\cap\alpha_{j}|$, for $j\ne1$.

\begin{figure}[h]
	\begin{minipage}[]{.45\linewidth}
		\centering
		\vspace{.2cm}
		\begin{lpic}[clean]{bigon1(,1.7cm)}
			\Large
			\lbl[]{98,6;$\beta_{0}$}
			\lbl[]{-8,14;$\beta_{0}'$}
			\lbl[]{14,-8;$\alpha_{1}$}
			\lbl[]{73,38;$\alpha_{j}$}
		\end{lpic}
		\vspace{.4cm}
		\caption{A point in $\alpha_{j}\cap\beta_{0}'$ corresponds to a point in $\alpha_{j}\cap\beta_{0}$.}
		\label{bigon1}
	\end{minipage}
	\hspace{.5cm}
	\begin{minipage}[]{.45\linewidth}
		\centering
		\vspace{.2cm}
		\begin{lpic}[clean]{bigon2(,1.5cm)}
			\Large
			\lbl[]{98,6;$\beta_{0}$}
			\lbl[]{-8,14;$\beta_{0}'$}
			\lbl[]{14,-8;$\alpha_{1}$}
			\lbl[]{79,37;$\alpha_{j}$}
		\end{lpic}
		\vspace{.4cm}
		\caption{`New' intersections of $\beta_{0}'$ with $\alpha_{j}$, violating the assumption that $\alpha_{1}$ and $\alpha_{j}$ are in minimal position.}
		\label{bigon2}
	\end{minipage}
\end{figure}

We apply these finitely many homotopies across bigons to $\beta_{0}$, one for each of the bigons formed by $\beta_{0}$ and $\alpha_{1}$. The result is a curve $\beta_{1}$, homotopic to $\beta$, so that $\beta_{1}$ and $\alpha_{1}$ are in minimal position, and so that $|\beta_{1}\cap\alpha_{j}|\le|\beta_{0}\cap\alpha_{j}|$, for $j\ne1$.  Do this one-by-one for each $\alpha_{k}$, and the result is a curve $\beta'$, homotopic to $\beta$, so that $\beta$ is in minimal position with $\alpha_{j}$ for each $j=1,\ldots,n$. Thus $\{\alpha_{1},\ldots,\alpha_{n},\beta'\}$ is an almost-realization.}\end{proof}

\begin{proof}[Proof of Proposition \ref{tool2claim}]{We proceed by induction on $n$. For $n=3$, choose a realization of the three curves that forms a triangle. Note that since the curves form a $1$-system, there is one boundary arc of the triangle contained in each of the curves. Fixing the curves outside of a disk containing the triangle, homotope one of the curves across the triangle so that it transversally crosses the other two curves at their intersection point. This produces an almost-realization of the curves, since no bigons have been created. 

Assume now that we have $n+1$ curves $\{\gamma_{1},\ldots,\gamma_{n+1}\}$ so that every trio forms a triangle. By the inductive hypothesis, there is an almost-realization $\{\alpha_{1},\ldots,\alpha_{n}\}$ of $\{\gamma_{1},\ldots,\gamma_{n}\}$ so that the $\alpha_{i}$ all have a unique common intersection point. Our strategy of proof will be to first choose a curve in the homotopy class of $\gamma_{n+1}$ that is in minimal position with the $\alpha_{i}$, and then to use the intersection properties of the curves -- namely the fact that each pair intersects exactly once, and that every trio forms triangles -- to `weave' this curve around the other curves, achieving the desired arrangement. Note that in this process the curves $\alpha_{i}$, for $i=1,\ldots,n$, remain unchanged.

Lemma \ref{homotopies to only one} implies that we may find a curve $\beta$, homotopic to $\gamma_{n+1}$, so that $\{\alpha_{1},\ldots,\alpha_{n},\beta\}$ is an almost-realization. After possibly renaming the curves $\alpha_{1},\ldots,\alpha_{n}$, we may assume that these curves are arranged around their unique common intersection in the counter-clockwise order $\alpha_{1},\ldots,\alpha_{n}$. Let $p_{0}$ indicate the common intersection point of the $\alpha_{i}$, and choose a disk neighborhood $C$ of $p_{0}$, embedded on $S$, so that $\alpha_{i}\cap C$ is connected for each $i=1,\ldots,n$.

Let $p_{i}:=\beta\cap\alpha_{i}$. Replacing $\beta$ with a homotopic curve if necessary, we may assume that all of the $p_{i}$ are contained in the arcs $\alpha_{i}\cap C$, and that $\beta\cap C$ consists only of arcs that intersect an arc $\alpha_{i}\cap C$. We now have a picture where $\beta$ weaves in and out of $C$, intersecting each of the arcs $\alpha_{i}\cap C$ precisely once. Outside of $C$ the curves $\{\alpha_{1},\ldots,\alpha_{n},\beta\}$ are disjoint. The argument that follows applies various homotopies to $\beta$, maintaining the fact that $\beta$ and $\alpha_{i}$ are in minimal position, for each $i=1,\ldots,n$. Note that as we apply such homotopies, the $p_{i}$ move as well.

In fact, for each $p_{i}$ there is a homotopy of the curve $\beta$ that slides the intersection $p_{i}$ along $\alpha_{i}$, out of $C$, until $p_{i}$ ``reappears'' on the other side of $p_{0}$ on $\alpha_{i}\cap C$. Since the curves $\alpha_{i}$ do not intersect outside of $C$, this homotopy leaves the collection $\{\alpha_{1},\ldots,\alpha_{n},\beta\}$, pairwise, in minimal position. This homotopy of $\beta$ will be exploited in the following argument, where we will refer to it as the `slide move applied to $p_{i}$'.

\begin{figure}
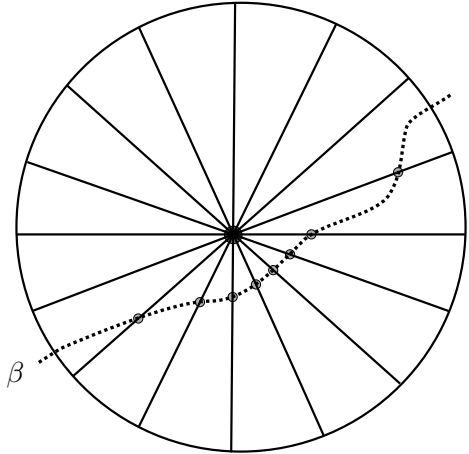

	\centering
	\begin{lpic}{greatPictureTriangles(,6cm)}
		\lbl[]{0,20;$\beta$}
	\end{lpic}
	\caption{A `localized realization' of $\beta$: $\beta\cap C$ is connected, and $p_{i}\in C$.}
	\label{localized realization}
\end{figure}

Our goal is to apply a homotopy to $\beta$ to achieve a special formation, in which intersections of $\beta$ with $\alpha_{i}$ have been `localized' to $C$: We will say a curve is a \emph{localized realization} of $\beta$ if it is homotopic to $\beta$, $\beta\cap C$ is connected, and $p_{i}\in C$, for each $i$. See Figure \ref{localized realization} for an illustrated example. The bulk of the proof of the inductive step concerns existence of a localized realization of $\beta$. We will proceed by analyzing the triangles formed by $\beta$ with $\alpha_{i}$ and $\alpha_{i+1}$, as $i$ goes from $1$ to $n-1$. At each step, we either find a localized realization of $\beta$, or we apply a homotopy to $\beta$ so that the number of connected components of $\beta\cap C$ is at most $n-i$.

Given vertices $a_{1},a_{2},a_{3}$ of the triangle $T$, we will say that $T$ `realizes $(a_{1},a_{2},a_{3})$' if the counter-clockwise orientation of $\partial T$ induces the cyclic order $(a_{1},a_{2},a_{3})$. We will refer to a sufficiently small neighborhood of $a_{i}$ as a `corner' of $T$, or `the corner at $a_{i}$' when the points are distinct, where the sufficiency is fulfilled when the intersection of $T$ with this neighborhood is connected and disjoint from the side opposite the vertex.  

Suppose that two curves $\eta$ and $\delta$ in minimal position, together with a third curve, form a triangle $T$, so that $T$ has a corner at the intersection point $p\in\eta\cap\delta$. In this case, the complement of $\eta\cup\delta$ in a small enough neighborhood of $p$ consists of four components. Note that these four components correspond to the four possibilities for the placement of a corner of $T$ at $p$. (When $\eta$ and $\delta$ are simple curves it is straightforward to check that these possibilities are mutually exclusive). This is exploited repeatedly below, for the placement of corners of triangles at $p_{0}$.

By hypothesis, and by Lemma \ref{triangles}, there is a triangular component $T_{1}$ of $S\setminus\{\alpha_{1},\alpha_{2},\beta\}$. Because the curves form a $1$-system, the vertices of $T_{1}$ are necessarily given by $p_{0}$, $p_{1}$, and $p_{2}$. Of the four ways for there to be a corner formed by $\alpha_{1}$ and $\alpha_{2}$ at $p_{0}$, two of them -- namely, the choice of corner so that $T$ realizes $(p_{0},p_{2},p_{1})$ (see Figure \ref{triangle 1}) -- would achieve a localized realization: The composition, if necessary, of slide moves applied to $p_{1}$ and $p_{2}$ with a homotopy of $\beta$ across the triangle $T_{2}$ would form a localized realization. We thus assume that $T$ realizes $(p_{0},p_{1},p_{2})$, in which case, applying slide moves to $p_{1}$ and $p_{2}$ if necessary, we have the situation pictured in Figure \ref{triangle 2}. Applying further slide moves when necessary, we may now assume that the $p_{i}$, for $i>2$, all lie on the same side of $\alpha_{1}\cap C$ inside $C$, as in Figure \ref{triangle 3}.

\begin{figure}
	\begin{minipage}[]{.45\linewidth}
	\centering
	\vspace{.5cm}
	\begin{lpic}[clean]{inductiveTriangles2(,6cm)}
		\Large
		\lbl[]{-2,88;$C$}
		\lbl[]{70,66;$p_{0}$}
		\lbl[]{48,28;$p_{1}$}
		\lbl[]{35,70;$p_{2}$}
		\lbl[]{100,5;$\alpha_{2}$}
		\lbl[]{100,110;$\alpha_{1}$}
		\lbl[]{35,120;$\beta$}
		\lbl[]{-8,45;$\alpha_{n}$}		
	\end{lpic}
		\caption{The triangle $T_{1}$ realizes $(p_{0},p_{2},p_{1})$, forming a localized realization of $\beta$.}
		\label{triangle 1}
	\end{minipage}
	\hspace{.5cm}
	\begin{minipage}[]{.45\linewidth}
	\centering
	\begin{lpic}[clean]{inductiveTriangles1(,5.8cm)}
		\Large
		\lbl[]{116,88;$C$}
		\lbl[]{55,70;$p_{0}$}
		\lbl[]{45,28;$p_{1}$}
		\lbl[]{65,28;$p_{2}$}
		\lbl[]{18,112;$\alpha_{2}$}
		\lbl[]{100,110;$\alpha_{1}$}
		\lbl[]{5,12;$\beta$}
		\lbl[]{125,68;$\alpha_{n}$}
	\end{lpic}
	\caption{The triangle $T_{1}$ realizes $(p_{0},p_{1},p_{2})$, where no localized realization is yet assured.}
	\label{triangle 2}
	\end{minipage}
\end{figure}

We now consider a triangular component $T_{2}$ of $S\setminus\{\alpha_{2},\alpha_{3},\beta\}$, with vertices $p_{0}$, $p_{2}$, and $p_{3}$. As before, there are four possible placements of the corner of $T_{2}$ at $p_{0}$. Two of them correspond to a situation in which $T_{2}$ realizes $(p_{0},p_{3},p_{2})$, again allowing a homotopy of $\beta$ that forms a localized realization. Assuming then that $T_{2}$ realizes $(p_{0},p_{2},p_{3})$, we now describe why there is only one possible placement of the corner of $T_{2}$ at $p_{0}$.

\begin{figure}
	\begin{minipage}[]{.45\linewidth}
	\vspace{.5cm}
	\centering
	\begin{lpic}[clean]{inductiveTriangles3(6cm)}
		\Large	
		\lbl[]{-2,88;$C$}
		\lbl[]{51,72;$p_{0}$}
		\lbl[]{32,39;$p_{1}$}
		\lbl[]{48,23;$p_{2}$}
		\lbl[]{77,41;$p_{3}$}
		\lbl[]{100,110;$\alpha_{1}$}
		\lbl[]{55,122;$\alpha_{2}$}
		\lbl[]{18,110;$\alpha_{3}$}
		\lbl[]{5,12;$\beta$}
		\lbl[]{-8,45;$\alpha_{n}$}
	\end{lpic}
	\caption{Given that the triangle $T_{1}$ realizes $(p_{0},p_{1},p_{2})$, the arcs $\beta\cap C$ may appear as pictured.}
	\label{triangle 3}
	\end{minipage}
	\hspace{.5cm}
	\begin{minipage}[]{.45\linewidth}
	\vspace{.8cm}
	\centering
	\begin{lpic}[clean]{inductiveTriangles4(6cm)}
		\Large
		\lbl[]{120,90;$C$}
		\lbl[]{49,80;$p_{0}$}
		\lbl[]{33,42;$p_{1}$}
		\lbl[]{62,35;$p_{2}$}
		\lbl[]{77,47;$p_{3}$}
		\lbl[]{100,115;$\alpha_{1}$}
		\lbl[]{56,129;$\alpha_{2}$}
		\lbl[]{15,118;$\alpha_{3}$}
		\lbl[]{4,13;$\beta$}
		\lbl[]{125,71;$\alpha_{n}$}
	\end{lpic}
	\caption{Given triangle $T_{1}$ realized as $(p_{0},p_{1},p_{2})$, the darkly shaded triangle $T_{2}$ cannot have a corner at $p_{0}$ as pictured.}
	\label{triangle 4}
	\end{minipage}
\end{figure}

First suppose the corner of $T_{2}$ at $p_{0}$ is placed as pictured in Figure \ref{triangle 4}. That is, suppose $T_{2}$ does not share the side between $p_{0}$ and $p_{2}$, contained in the arc $\alpha_{2}\cap C$, with $T_{1}$. Because $T_{2}$ realizes $(p_{0},p_{2},p_{3})$, the counter-clockwise orientation of $\partial T_{2}$ turns left from $\alpha_{2}$ to $\beta$ at $p_{2}$. There are two ways to make a left turn from $\alpha_{2}$ to $\beta$ at $p_{2}$. Of them, the one so that $T_{2}$ does not share the side between $p_{0}$ and $p_{2}$ with $T_{1}$ contains the arc of the curve $\beta$ that lies between the vertices $p_{2}$ and $p_{3}$ and containing $p_{1}$. Consequently, one side of the arc $\alpha_{1}\cap C$ leaves $C$ inside $T_{2}$, while the other leaves $C$ outside of $T_{2}$. Since $\alpha_{1}$ may only intersect the sides of $T_{2}$ inside $C$, this is impossible.

This leaves only one possibility for the placement of the corner of $T_{2}$ at $p_{0}$ in which $T_{2}$ realizes $(p_{0},p_{2},p_{3})$ (see Figure \ref{triangle 5}). This case allows us to apply a homotopy to $\beta$ supported in the triangle $T_{2}$, leaving the intersections $p_{i}$ inside $C$, and ensuring that the number of connected components of $\beta\cap C$ is at most $n-2$ (see Figure \ref{triangle 6}). 

Similarly, at the $k$th step, the triangular component $T_{k}$ of $S\setminus\{\alpha_{k},\alpha_{k+1},\beta\}$ either provides a homotopy of $\beta$ that achieves a localized realization, or provides a homotopy of $\beta$, supported inside $C$, that ensures that the number of connected components of $\beta\cap C$ is at most $n-k$. When $k=n-1$, we have ensured the existence of a localized realization of $\beta$.

\begin{figure}
	\begin{minipage}[]{.45\linewidth}
	\vspace{.5cm}
	\centering
	\begin{lpic}[clean]{inductiveTriangles5(6cm)}
		\Large
		\lbl[]{-2,88;$C$}
		\lbl[]{49,80;$p_{0}$}
		\lbl[]{33,42;$p_{1}$}
		\lbl[]{46,18;$p_{2}$}
		\lbl[]{77,42;$p_{3}$}
		\lbl[]{100,115;$\alpha_{1}$}
		\lbl[]{56,129;$\alpha_{2}$}
		\lbl[]{15,118;$\alpha_{3}$}
		\lbl[]{4,13;$\beta$}
		\lbl[]{-8,45;$\alpha_{n}$}
	\end{lpic}
	\caption{The remaining case where triangle $T_{2}$ realizes $(p_{0},p_{2},p_{3})$.}
	\label{triangle 5}
	\end{minipage}
	\hspace{.7cm}
	\begin{minipage}[]{.45\linewidth}
	\vspace{.5cm}
	\centering
	\begin{lpic}[clean]{inductiveTriangles6(6cm)}
		\Large
		\lbl[]{120,90;$C$}
		\lbl[]{49,80;$p_{0}$}
		\lbl[]{33,42;$p_{1}$}
		\lbl[]{46,18;$p_{2}$}
		\lbl[]{77,42;$p_{3}$}
		\lbl[]{100,115;$\alpha_{1}$}
		\lbl[]{56,125;$\alpha_{2}$}
		\lbl[]{15,118;$\alpha_{3}$}
		\lbl[]{4,13;$\beta$}
		\lbl[]{125,68;$\alpha_{n}$}
	\end{lpic}
	\caption{Applying a homotopy to $\beta$ using $T_{2}$, getting `closer' to a localized realization.}
	\label{triangle 6}
	\end{minipage}
\end{figure}

Finally, after replacing $\beta$ with its localized realization, we assume that $\beta\cap C$ is connected, and that $p_{i}\in C$ for each $i$. Note that this implies that $\beta\cap\partial C$ consists of exactly two points, which we denote $p_{-}$ and $p_{+}$. A straightforward application of the Jordan Curve Theorem ensures that $p_{-}$ and $p_{+}$ are in diametrically opposed components of $\partial C \setminus \{\alpha_1,\ldots,\alpha_n\}$, and we may apply a homotopy to $\beta$ making it into a diameter passing through $p_{0}$, completing the inductive step.}
\end{proof} 

A word of caution: It is not generally true that if $\Gamma'\subset\Gamma$ then the complex $C(\Gamma')$ is a subcomplex of $C(\Gamma)$ (see Figure \ref{3 simple curves}). The following corollary is a weaker version of such a statement that will suffice for our application. Given a collection of curves $\Gamma$, we will say that a subset of $n$ curves $\Gamma'\subset\Gamma$ \emph{form an n-cube} in $C(\Gamma)$ if there are $n$ hyperplanes corresponding to the curves of $\Gamma'$ intersecting in an $n$-cube of $C(\Gamma)$. 

\begin{corollary}
\label{3-to-n subset}
{If $\Gamma$ is a $1$-system of curves, and $\{\gamma_{1},\ldots,\gamma_{n}\}=\Gamma'\subset\Gamma$, then the curves of $\Gamma'$ form an $n$-cube in $C(\Gamma)$ if and only if every triple of curves from $\Gamma'$ form a $3$-cube.}
\end{corollary}

\begin{figure}[h]
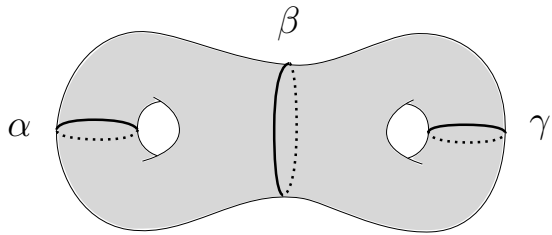

\vspace{.5cm}
	\begin{lpic}[clean]{3SimpleCurves(6cm)}
		\Large
		\lbl{-12,35;$\alpha$}
		\lbl{78,70;$\beta$}
		\lbl{162,35;$\gamma$}
	\end{lpic}
	\caption{The complex $C(\{\alpha,\gamma\})$ is not a subcomplex of $C(\{\alpha,\beta,\gamma\})$.}
	\label{3 simple curves}
\end{figure}

\begin{proof}
{Using the curves to hyperplanes correspondence, the curves of $\Gamma'$ form an $n$-cube in $C(\Gamma)$ if and only if there is a choice of lifts $\{\widetilde{\gamma_{1}},\ldots,\widetilde{\gamma_{n}}\}$ so that these lifts pairwise intersect, which in turn occurs if and only if $\dim C(\Gamma')=n$, at which point we apply Theorem \ref{3-to-n}.}
\end{proof} 

\section{A family of maximum complete 1-systems}
\label{construction section}

\begin{figure}[h]
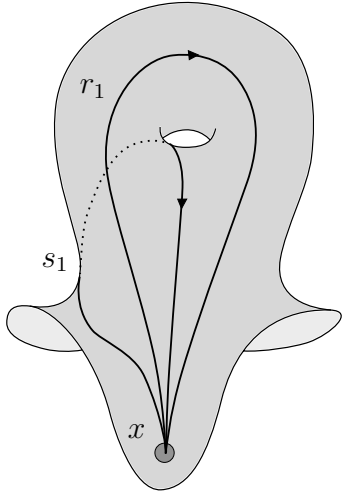

	\begin{lpic}[clean]{generatorsHandle(,6.5cm)}
		\large
		\lbl{120,55;$x$}
		\lbl{80,370;$r_{1}$}
		\lbl{45,210;$s_{1}$}
	\end{lpic}
	\caption{The generators of $\pi_{1}(S,x)$ in one handle.}
	\label{generators handle1}
\end{figure}

We now construct many maximum complete 1-systems on a surface $S$ of any odd genus $g=2n+1$. Consider the genus $g$ surface as the unit sphere in $\R^{3}$ with $g$ handles attached at evenly spaced disks centered on an equator, making the order $g$ homeomorphism $\sigma$ that cyclically permutes the handles apparent. Let $x$ be a point fixed by this homeomorphism. Consider the presentation $$\pi_{1}(S,x)=\left\langle r_{i},s_{i} \left| \; \prod_{i=1}^{g}[r_{i},s_{i}] \right. \right\rangle,$$ where the generators $r_{1}$ and $s_{1}$ are as pictured in \ref{generators handle1}, and $r_{i}=\sigma(r_{i-1})$ and $s_{i}=\sigma(s_{i-1})$ for $i=2,\ldots,g$.

Let $\alpha_{1}$, $\beta_{1}$, and $\delta$ be given by 
\begin{enumerate}
\item $\alpha_{1}=\left[r_{n+2}r_{n+3}\ldots r_{2n+1}s_{1}\right]$, 
\item $\beta_{1}=\left[r_{n+2}r_{n+3}\ldots r_{2n+1}s_{1}^{-1}\displaystyle\prod_{i=1}^{n+1}[s_{i},r_{i}]\right]$, and
\item $\delta=\left[r_{1}r_{2}\ldots r_{2n+1}\right]$.
\end{enumerate}

The orbit of $\alpha_{1}$ under $\sigma$ gives $g$ curves which we denote by $\{\alpha_{1},\alpha_{3},\ldots,\alpha_{2g-1}\}$. Similarly, we denote the $\sigma$-orbit of $\beta_{1}$ by $\{\beta_{1},\beta_{3},\ldots,\beta_{2g-1}\}$. We complete these collections to sequences $\{\alpha_{i}\}_{i=1}^{2g}$ and $\{\beta_{i}\}_{i=1}^{2g}$ by defining $\alpha_{2i}=\tau_{i}(\alpha_{2i-1})$ and $\beta_{2i}=\tau_{i}^{-1}(\beta_{2i-1})$ for $i=1,\ldots,g$, and where $\tau_{i}$ is the right Dehn twist around $r_{i}$. See Figures \ref{alphaCurve}, \ref{betaCurve}, \ref{moreAlphaBeta}, and \ref{deltaCurve} for illustrative examples.

Grouping these curves together, we will refer to $A=\{\alpha_{1},\ldots,\alpha_{2g}\}$ as the set of `up' curves and $B=\{\beta_{1},\ldots,\beta_{2g}\}$ as the set of `down' curves.  We will refer to the pair of up curves (resp.~down curves) $\alpha_{2i-1}$ and $\alpha_{2i}$ (resp.~$\beta_{2i-1}$ and $\beta_{2i}$) as `partners', for $i=1,\ldots,g$. 

\begin{figure}[h]
	\begin{minipage}[]{.48\linewidth}
		\centering
		\includegraphics[width=5.5cm]{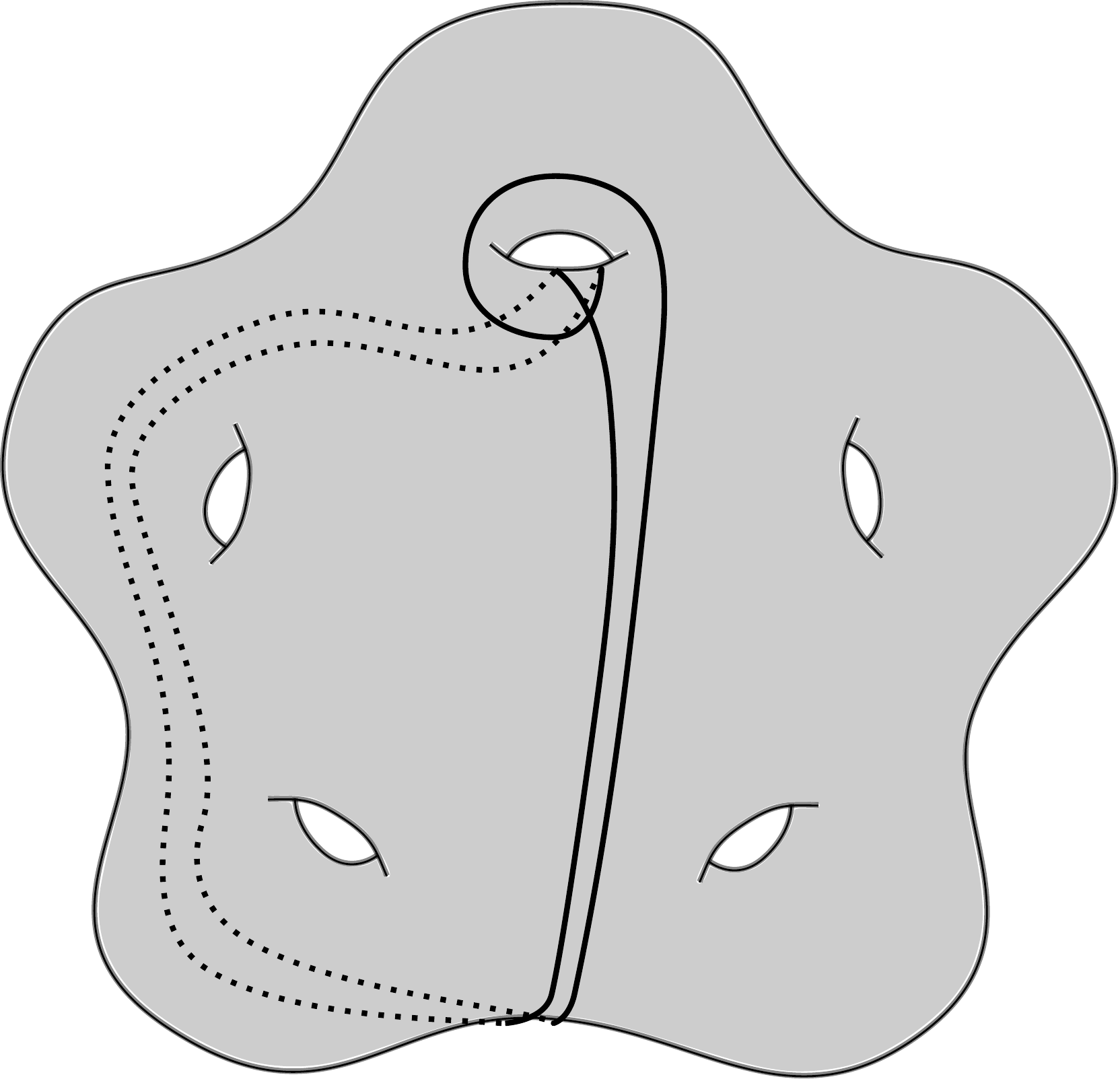}
		\caption{The curves $\alpha_{1},\alpha_{2}\in A$}
		\label{alphaCurve}
	\end{minipage}
	\begin{minipage}[]{.48\linewidth}
		\centering
		\includegraphics[width=5.5cm]{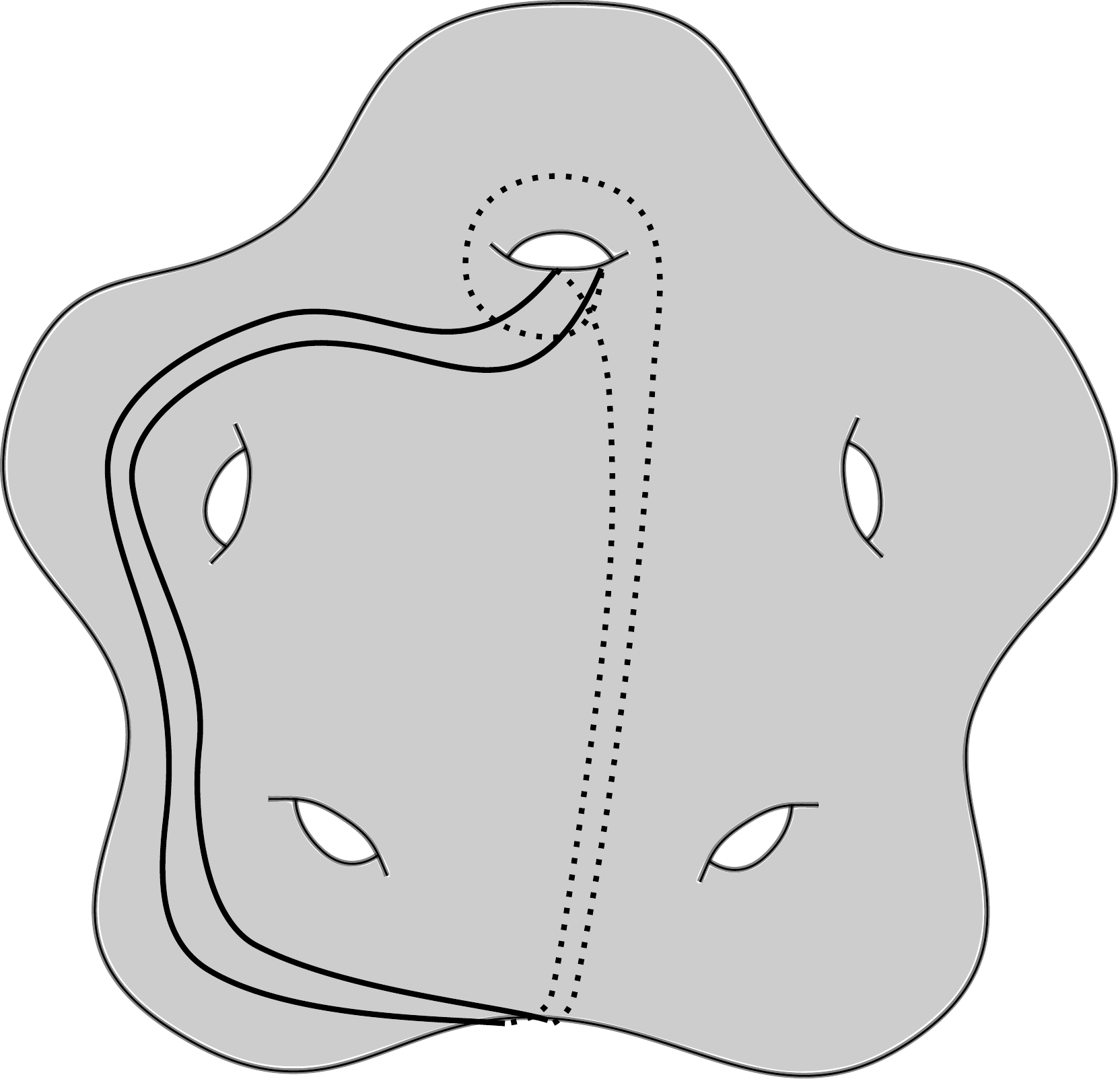}
		\caption{The curves $\beta_{1},\beta_{2}\in B$}
		\label{betaCurve}
	\end{minipage}
\end{figure}
\begin{figure}[h]
	\begin{minipage}[]{.48\linewidth}
		\centering
		\includegraphics[width=5.5cm]{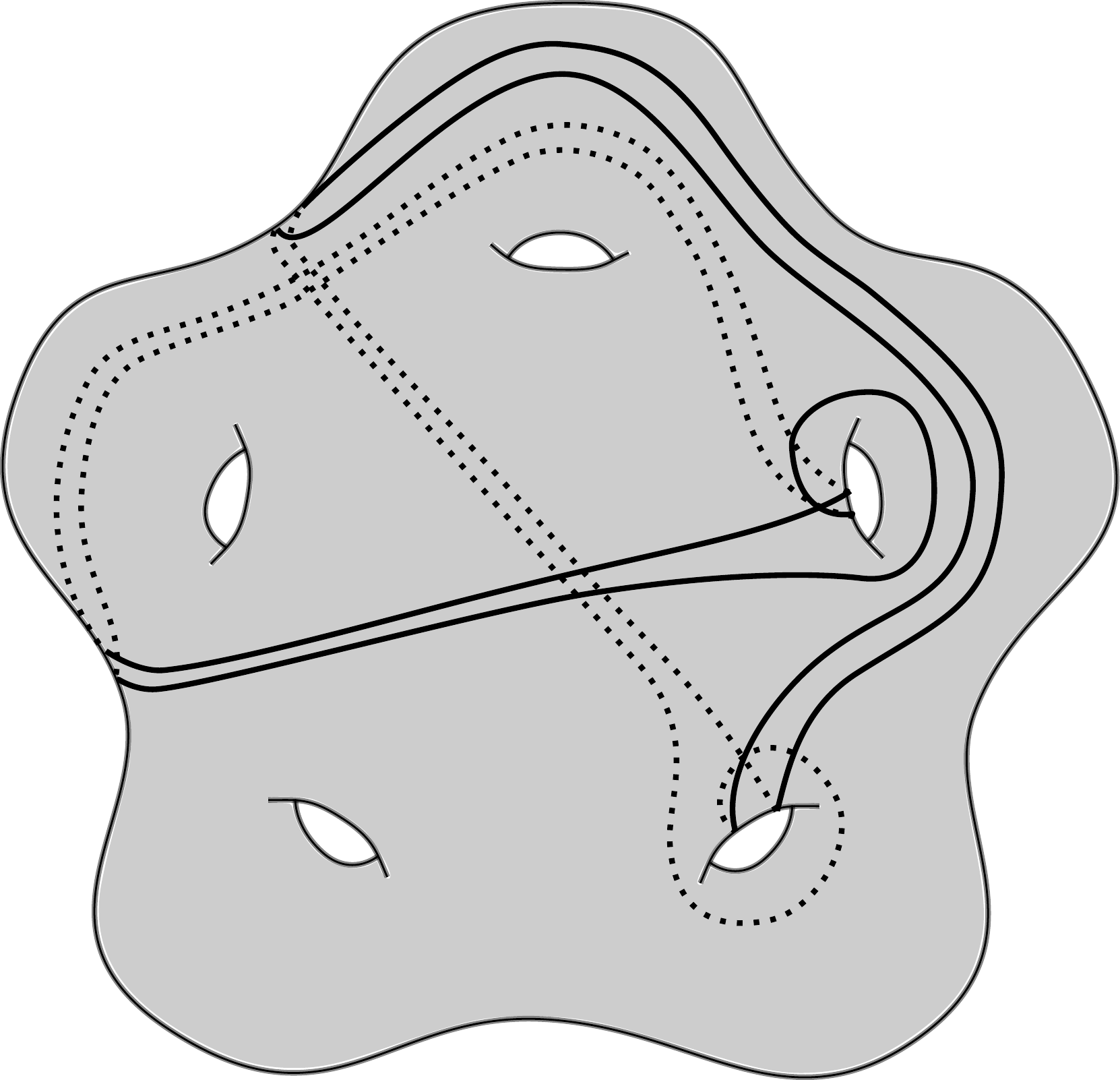}
		\caption{The curves $\alpha_{3},\alpha_{4}\in A$ and $\beta_{5},\beta_{6}\in B$}		
		\label{moreAlphaBeta}
	\end{minipage}
	\begin{minipage}[]{.48\linewidth}
		\centering
		\includegraphics[width=5.5cm]{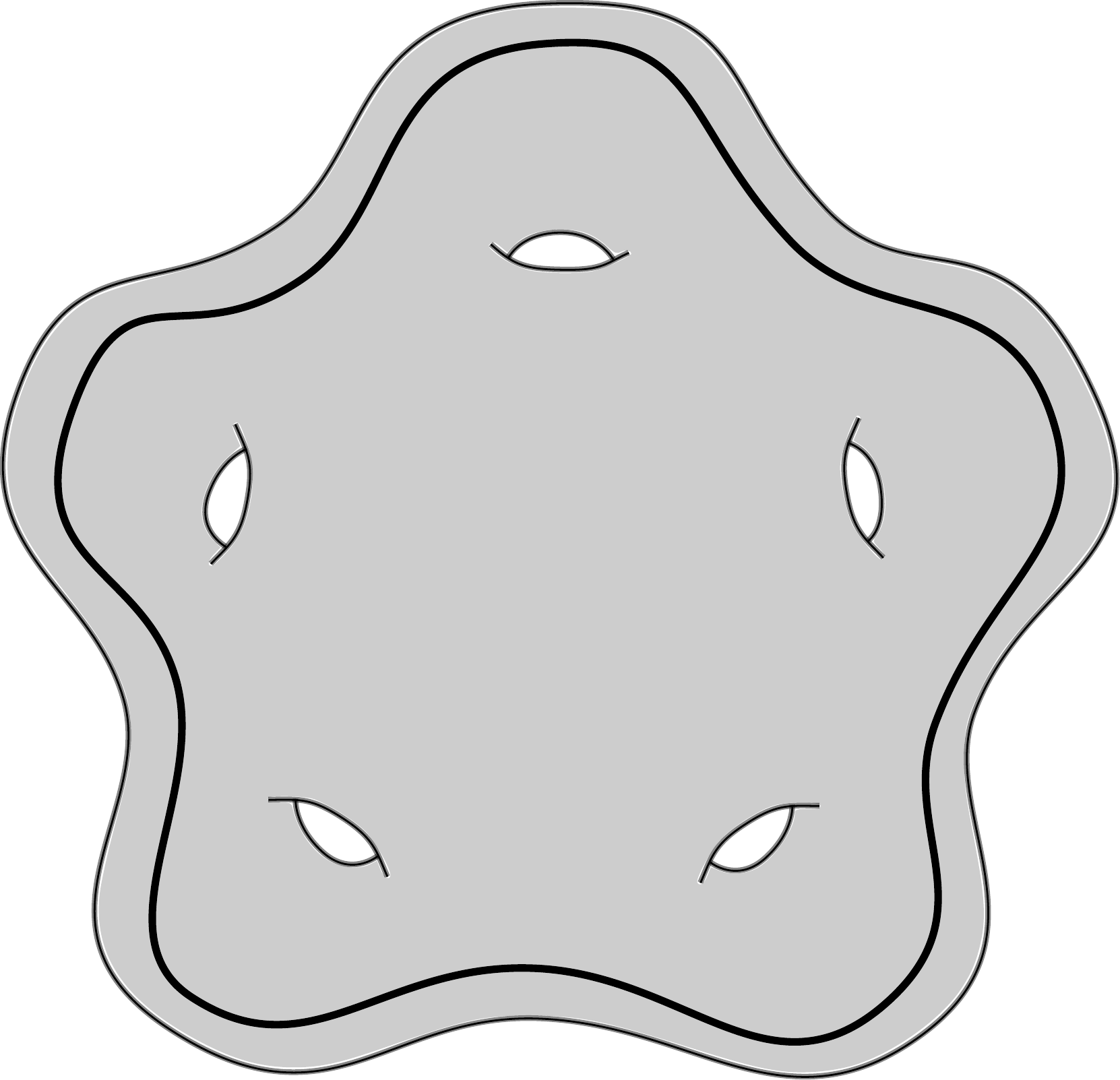}
		\caption{The curve $\delta$}
		\label{deltaCurve}
	\end{minipage}
\end{figure}

It is immediate that, $\alpha_{i}\cap\alpha_{j}=\beta_{i}\cap\beta_{j}=1$ for $i\ne j$. When $\alpha_{i}$ and $\alpha_{j}$ are not partners, then $\alpha_{i}\cap\beta_{j}=1$. This calculation makes it clear that we may form many maximum complete $1$-systems: For each $i=1,\ldots,g$, choose one of the two pairs of partners from $\{\alpha_{2i-1},\alpha_{2i}\}$ and $\{\beta_{2i-1},\beta_{2i}\}$. Together with the $\delta$ curve, this forms $2g+1$ curves that pairwise intersect once. This is maximum by \cite[Thm.~1.4]{m-r-t}.

\begin{figure}[h]
	\centering
	\includegraphics[height=8.5cm]{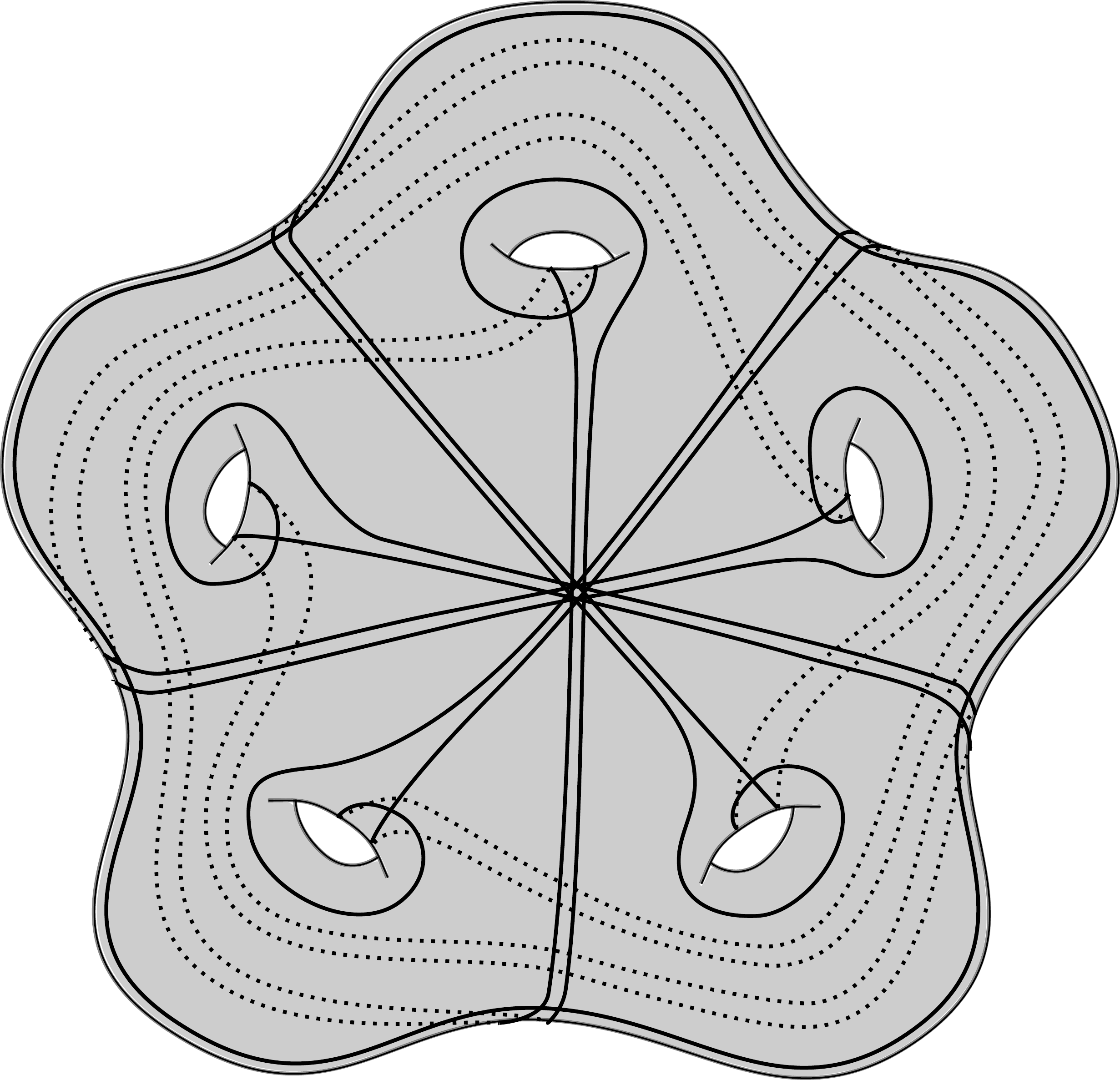}
	\caption{The system of curves $\Gamma(1,1,1,1,1)$}
	\label{Gamma(1,1,1,1,1)}
\end{figure}

\begin{figure}[h]
	\centering
	\includegraphics[height=8.5cm]{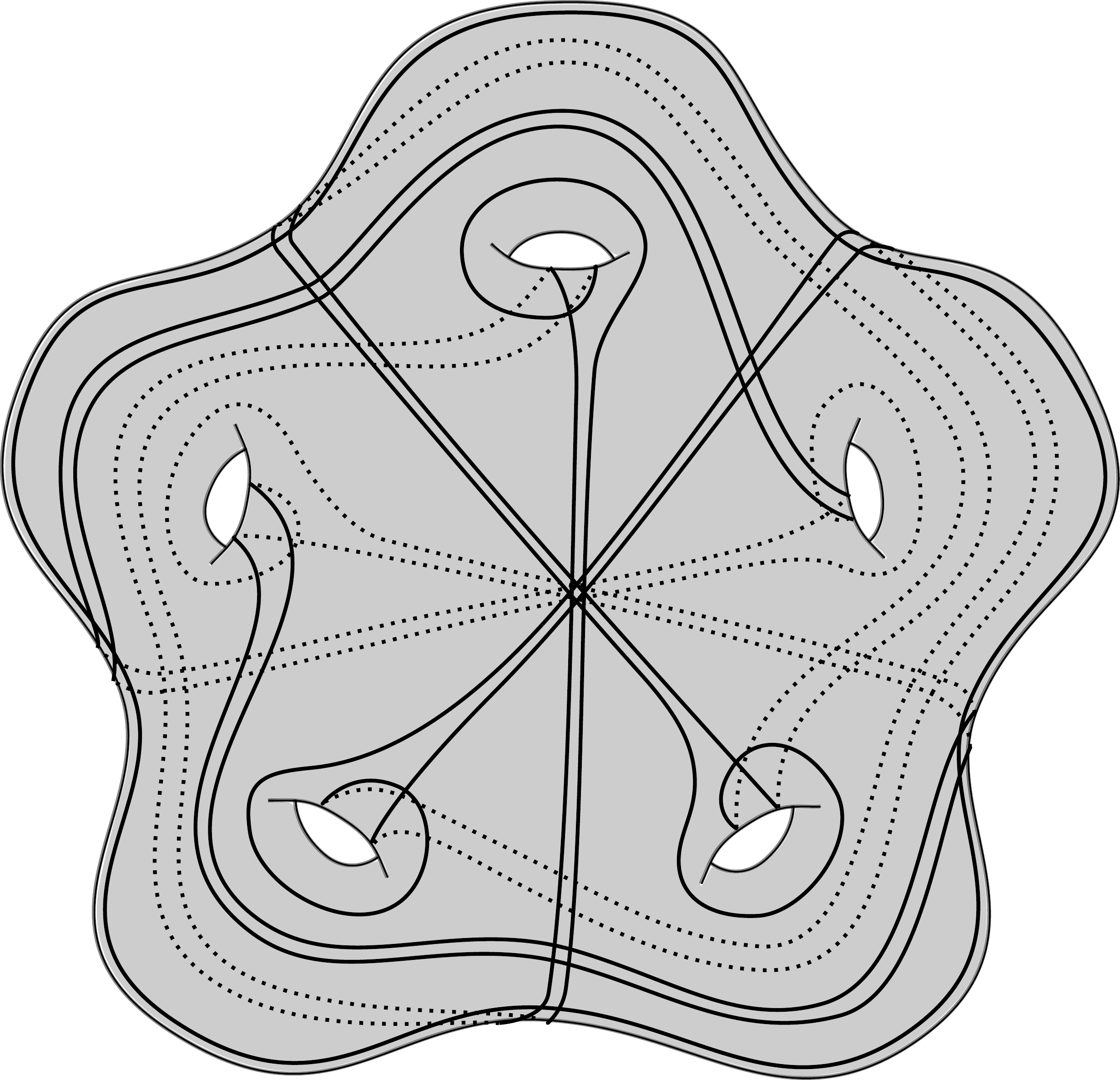}
	\caption{The system of curves $\Gamma(1,-1,1,1,-1)$}
	\label{Gamma(1,-1,1,1,-1)}
\end{figure}

\begin{figure}[h]
	\centering
	\includegraphics[height=8.5cm]{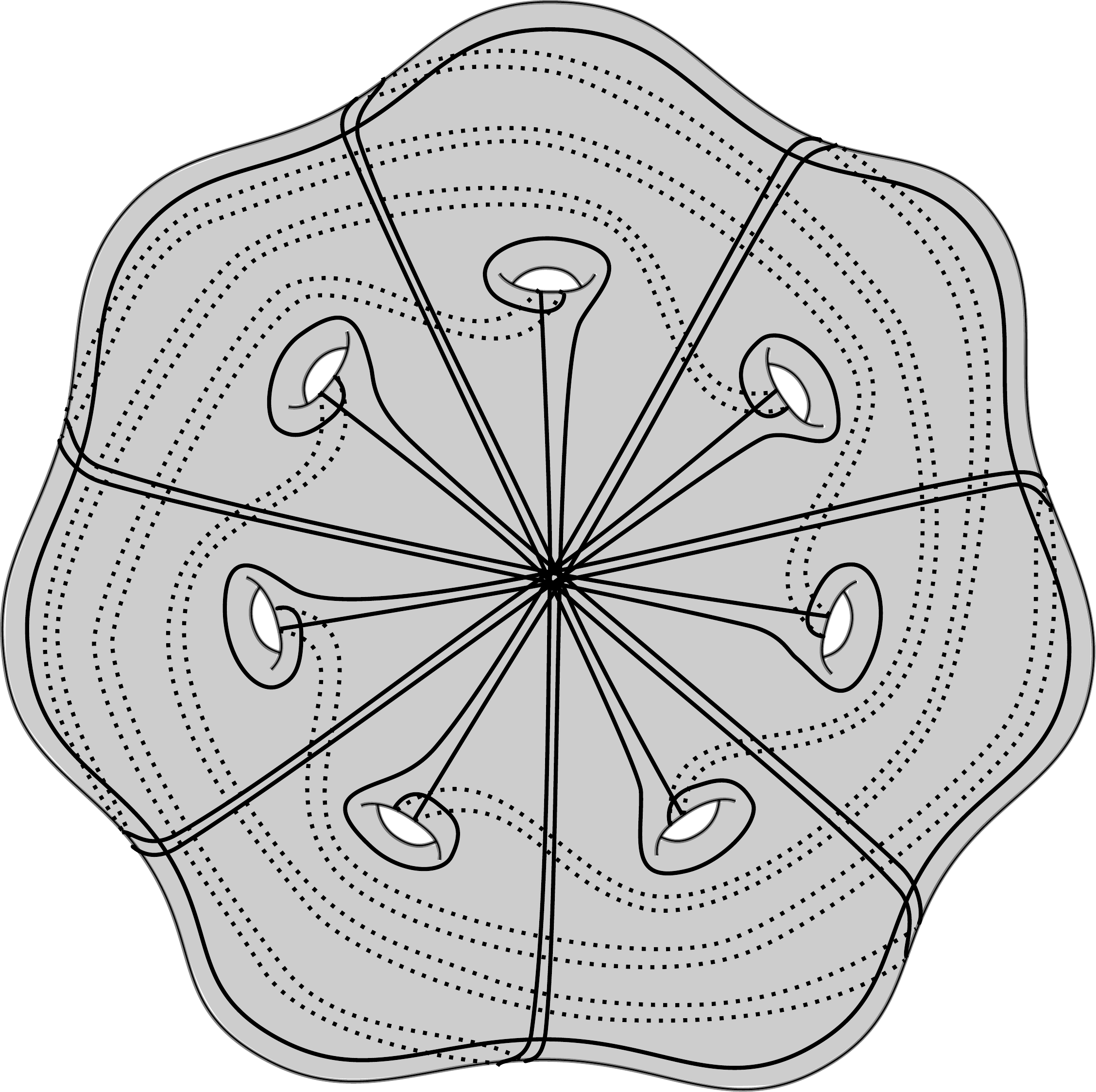}
	\caption{The system of curves $\Gamma(1,1,1,1,1,1,1)$}
	\label{Gamma(1,1,1,1,1,1,1)}
\end{figure}
\begin{figure}[h]
	\centering
	\includegraphics[height=8.5cm]{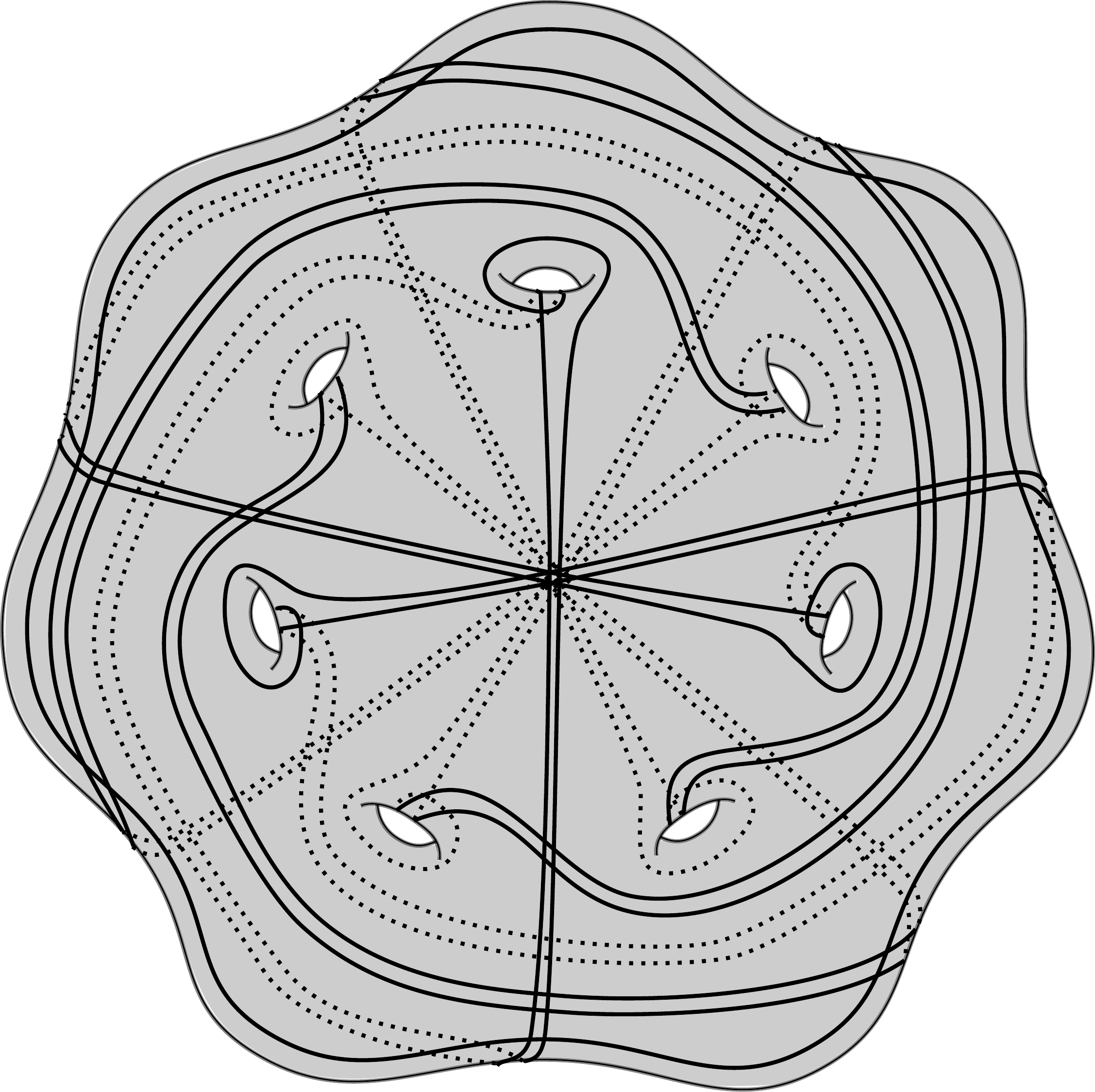}
	\caption{The system of curves $\Gamma(1,-1,1,-1,-1,1,-1)$}
	\label{Gamma(1,-1,1,-1,-1,1,-1)}
\end{figure}

More precisely, for $\epsilon\in\{1,-1\}^{g}$, let $A(\epsilon)=\{\alpha_{2i-1},\alpha_{2i}|\epsilon_{i}=1\}$ and $B(\epsilon)=\{\beta_{2i-1},\beta_{2i}|\epsilon_{i}=-1\}$, and let $\Gamma(\epsilon)=\{\delta\}\cup A(\epsilon)\cup B(\epsilon)$. Several examples in genus 5 and 7 are shown in Figures \ref{Gamma(1,1,1,1,1)}, \ref{Gamma(1,-1,1,1,-1)}, \ref{Gamma(1,1,1,1,1,1,1)}, and \ref{Gamma(1,-1,1,-1,-1,1,-1)}.

\begin{lemma}
\label{examples mc1s}
{For each $\epsilon\in\{-1,1\}^{g}$, the collection of curves $\Gamma(\epsilon)$ forms a maximum complete 1-system.}
\end{lemma}

There is an action of the dihedral group $D_{g}$ on $\{1,-1\}^{g}$ given by letting the generators act by a $g$-cycle and a reversal of the list, respectively. Letting $\Z/2\Z$ act by taking $\epsilon$ to $-\epsilon$, we obtain an action of $\Z/2\Z\oplus D_{g}$ on $\{1,-1\}^{g}$, and there is naturally an action of $\mbox{Mod}^{*}(S)$ on systems of conjugacy classes of curves on $S$. The following proposition, whose proof occupies the bulk of our analysis in \S\ref{using dual cube complex section} and \S\ref{polygon section}, implies that the maximum complete 1-systems from Lemma \ref{examples mc1s} represent many distinct orbits.  

\begin{proposition}
\label{distinguishing orbits}
{If $\Gamma(\epsilon)$ and $\Gamma(\epsilon')$ are in the same $\mbox{Mod}^{*}(S)$-orbit, then $\epsilon$ and $\epsilon'$ are in the same $(\Z/2\Z\oplus D_{g})$-orbit in $\{1,-1\}^{g}$.}
\end{proposition}

We will also require the simple observation:

\begin{lemma}
\label{mc1s filling}
{A maximum complete $1$-system is filling.}
\end{lemma}

\begin{proof}
{Suppose not. Then there is a simple closed curve $\alpha$ disjoint from the curves in our $1$-system. Cut open $S$ along $\alpha$, and cap off the two resulting boundary components created with disks. Note that the resulting surface may be disconnected. In any case, the set of $2g+1$ curves obtained forms a maximum complete $1$-system of curves on a surface of genus $g'\le g-1$, contradicting \cite[Thm.~1.4]{m-r-t}.}
\end{proof}

In fact, one can show that any $2g$ curves from a maximum complete 1-system are filling, but we will not require this stronger statement.

Let $N(g)$ indicate the number of $\mbox{Mod}^{*}(S)$-orbits among maximum complete 1-systems. A simple argument involving the square complex dual to a realization of curves on $S$ provides an upper bound for $N(g)$ below. For $g$ odd, Proposition \ref{distinguishing orbits} allows using the $\Gamma(\epsilon)$ to provide a lower bound for $N(g)$. For $g$ even, we will use the stabilization procedure described in detail in \S\ref{stabilizing} to obtain lower bounds for $N(g)$. Given a choice of realization for a maximum complete 1-system $\lambda$, and an arc $\alpha$ intersecting each of the curves in $\lambda$ once, stabilization produces a maximum complete 1-system on a surface of genus $g+1$.

When necessary below, we identify $\Gamma(\epsilon)$ with fixed choices of realization for each such collection. In \S\ref{stabilizing} we prove:

\begin{proposition}
\label{stabilizing Gammas}
{There is a choice of arc $\alpha$ so that the stabilizations of $\Gamma(\epsilon)$ and $\Gamma(\epsilon')$ along $\alpha$ are in the same $\mbox{Mod}^*(S_{g+1})$-orbit if and only if $\Gamma(\epsilon)$ and $\Gamma(\epsilon')$ are in the same $\mbox{Mod}^*(S_{g})$-orbit.}
\end{proposition}

We restate and prove Theorem \ref{MainThm}:

\begin{customthm}{1}
{We have the bounds $$(4g^{2}+2g)! \ge N(g) \ge \frac{2^{g-1}}{4(g-1)}.$$}
\end{customthm}

\begin{proof}
{The lower bound follows from Proposition \ref{distinguishing orbits} when $g$ is odd, and from Proposition \ref{stabilizing Gammas} when $g$ is even.

Towards the upper bound, consider the set of isomorphism classes of square complexes $\mathcal{S}_{\lambda}$ that are dual to realizations $\lambda$ of maximum complete $1$-systems. For filling systems of curves, the dual square complex is isomorphic to the surface $S$, and the hyperplanes of the square complex are in the homotopy classes of the curves one started with. Lemma \ref{mc1s filling} now guarantees that an isomorphism of square complexes $\mathcal{S}_{\lambda}\cong\mathcal{S}_{\lambda'}$ yields a homeomorphism of $S$ taking $\lambda$ to $\lambda'$. Thus there is a well-defined map from the set of isomorphism classes of square complexes dual to maximum complete $1$-systems to the set of $\mbox{Mod}^{*}(S)$-orbits of maximum complete $1$-systems. This map is evidently surjective, so that an upper bound for the number of possible square complexes dual to a realization of a maximum complete 1-system produces an upper bound for $N(g)$. 

For each realization $\lambda$ of a maximum complete $1$-system, each of the curves in $\lambda$ passes through exactly $2g$ squares of $\mathcal{S}_{\lambda}$. We may thus view $\mathcal{S}_{\lambda}$ as the quotient of the disjoint union of $2g+1$ annuli, each of which is built from $2g$ squares, where the quotient map identifies squares in pairs. There are at most $\binom{2g(2g+1)}{2,\ldots,2}$ pairings, and each pair of matched squares has two possible identifications, giving at most
$$2^{g(2g+1)}\binom{2g(2g+1)}{2,\ldots,2}=2^{g(2g+1)}\cdot\frac{\left(2g(2g+1)\right)!}{2^{g(2g+1)}}=\left(2g(2g+1)\right)!$$
square complexes $\mathcal{S}_{\lambda}$.}
\end{proof}

\section{Restricting mapping class group orbits via $C(\Gamma(\epsilon))$}
\label{using dual cube complex section}
This section is the first step towards the proof of Proposition \ref{distinguishing orbits}.

\begin{proposition}
\label{MCG-orbit restrictions}
{If $\phi\in \mbox{Mod}^{*}(S)$ satisfies $\phi\cdot\Gamma(\epsilon)=\Gamma(\epsilon')$, then 
\begin{enumerate}
\item $\phi \cdot \delta =\delta$, and
\item $\phi\cdot \left\{A(\epsilon),B(\epsilon)\right\}=\left\{A(\epsilon'),B(\epsilon')\right\}$.
\end{enumerate}
Moreover, the map $\phi$ sends partner curves to partner curves.}
\end{proposition}
In other words, either $\phi$ preserves the sets of up and down curves, or it exchanges them. The proof of this proposition will follow from a coarse picture of the cube complex $C(\Gamma(\epsilon))$. 

\begin{lemma}
\label{3cubes}
{In the complex $C(A\cup B\cup\{\delta\})$, the triples that form $3$-cubes are the following:
\begin{enumerate}
\item $\{\alpha_{i},\alpha_{j},\alpha_{k}\}$ or $\{\beta_{i},\beta_{j},\beta_{k}\}$, for distinct $i,j,k\in\{1,\ldots,2g\}$.
\item $\{\alpha_{2i-1},\alpha_{2i},\delta\}$ or $\{\beta_{2i-1},\beta_{2i},\delta\}$, for $i\in\{1,\ldots,g\}$.
\item $\{\alpha_{2j-1},\alpha_{2j},\beta_{i}\}$ or $\{\beta_{2j-1},\beta_{2j},\alpha_{i}\}$, for $i\in\{1,\ldots,2g\}$ and $j\in\{1,\ldots,g\}$.
\item $\{\alpha_{i},\beta_{j},\delta\}$ for $i,j\in\{1,\ldots,2g\}$.
\end{enumerate}
}
\end{lemma}

\begin{proof}
{Using Lemmas \ref{triangles} and \ref{tool1}, we determine whether a triple of curves forms a $3$-cube by choosing a realization of the curves, and observing whether there is a triangular component of the complement. For each of the curves, we fix choices of realizations as in Figures \ref{alphaCurve}, \ref{betaCurve}, and \ref{deltaCurve}.

If $\delta$ is one of the three curves, we arrange the possible ways to choose the other two curves according to whether the curves are chosen as `up' or `down' (i.e.~from $A$ or $B$): If both of the other curves are up, then there is a triangle in the complement of the trio if and only if the other two curves were partners. If one of the curves is up and one is down, there is such a triangle. The other cases are similar.

On the other hand, if $\delta$ is not one of the three curves: If all of the curves are up, there is such a triangle. If two of the curves are up and one is down, there is a triangle in their complement if and only if the two up curves are partners. The other cases are similar.}
\end{proof}

We proceed with an examination of hyperplanes of maximal cubes in the cases for $\epsilon\in\{1,-1\}^{g}$ where $|A(\epsilon)|,|B(\epsilon)|>1$.

\begin{lemma}
\label{max cubes}
{When $|A(\epsilon)|,|B(\epsilon)|>1$, the sets of hyperplanes of maximal cubes of $C(\Gamma(\epsilon))$ correspond to one of the following lists of curves:

\begin{enumerate}
\item The $2|\epsilon^{-1}(1)|$ curves $A(\epsilon)$.
\item The $2|\epsilon^{-1}(-1)|$ curves $B(\epsilon)$.
\item The 5 curves $\{\alpha_{2i-1},\alpha_{2i},\beta_{2j-1},\beta_{2j},\delta\}$, for $i,j\in\{1,\ldots,g\}$ such that $\alpha_{2i}\in A(\epsilon)$ and $\beta_{2j}\in B(\epsilon)$.
\end{enumerate}}
\end{lemma}

\begin{proof}
{Using Corollary \ref{3-to-n subset}, in order to check whether a subset of curves from $\Gamma(\epsilon)$ forms an $n$-cube, it is enough to check whether every triple forms a $3$-cube. By Lemma \ref{3cubes} we have a complete list of such 3-cubes. 

The curves $A(\epsilon)$ and $B(\epsilon)$ form cubes of dimensions $2|\epsilon^{-1}(1)|$ and $2|\epsilon^{-1}(-1)|$, respectively, by Lemma \ref{3cubes} and Corollary \ref{3-to-n subset}. If one adds a down curve $\beta_{i}$ to $A(\epsilon)$, then a pair of up curves that are not partners will not form a $3$-cube with this down curve $\beta_{i}$, by Lemma \ref{3cubes}. If one adds $\delta$ to $A(\epsilon)$, then again a pair of up curves that are not partners will not form a $3$-cube with $\delta$. The analogous statements hold for $B(\epsilon)$. By Corollary \ref{3-to-n subset}, the cubes of dimension $2|\epsilon^{-1}(1)|$ and $2|\epsilon^{-1}(-1)|$ containing these sets of hyperplanes, respectively, must be maximal. The same analysis shows that a maximal cube containing $\delta$ must contain a pair of partner up curves and a pair of partner down curves.}
\end{proof}

The cases in which either of $|A(\epsilon)|$ or $|B(\epsilon)|$ are less than or equal to $1$ are quite similar, so we list the relevant information without proof.

\begin{lemma}
\label{max cubes 2}
{When $|B(\epsilon)|=0$ (resp.~$|A(\epsilon)|=0$), the sets of hyperplanes of maximal cubes of $C(\Gamma(\epsilon))$ correspond to one of the following lists of curves:

\begin{enumerate}
\item The $2g$ curves $A(\epsilon)$ (resp.~$B(\epsilon)$).
\item The 3 curves $\{\alpha_{2i-1},\alpha_{2i},\delta\}$ (resp.~$\{\beta_{2i-1},\beta_{2i},\delta\}$), for $i\in\{1,\ldots,g\}$.
\end{enumerate}
When $|B(\epsilon)|=1$ (resp.~$|A(\epsilon)|=1$), the sets of hyperplanes of maximal cubes of $C(\Gamma(\epsilon))$ correspond to one of the following lists of curves:

\begin{enumerate}
\item The $2g-2$ curves $A(\epsilon)$ (resp.~$B(\epsilon)$).
\item The 5 curves $\{\alpha_{2i-1},\alpha_{2i},\beta_{2j-1},\beta_{2j},\delta\}$, for $i,j\in\{1,\ldots,g\}$ such that $\alpha_{2i}\in A(\epsilon)$ and $\beta_{2j}\in B(\epsilon)$.\hfill \qedsymbol
\end{enumerate}}
\end{lemma}

\begin{proof}[Proof of Proposition \ref{MCG-orbit restrictions}]
{The simple observation we exploit is that cube complex isomorphisms must send maximal cubes to maximal cubes.

Suppose $|A(\epsilon)|,|B(\epsilon)|>1$. By Lemma \ref{max cubes} and Lemma \ref{max cubes 2}, the maximal cubes that the hyperplane corresponding to $\delta$ passes through are all $5$-dimensional, while the maximal cubes that the hyperplane corresponding to $\alpha_{i}$ (resp.~$\beta_{i}$) passes through, for any $i\in\{1,\ldots,2g\}$, include one of the two even-dimensional maximal cubes corresponding to $A(\epsilon)$ and $B(\epsilon)$. Thus any isomorphism of cube complexes $\Phi:C(\Gamma(\epsilon))\cong C(\Gamma(\epsilon'))$ must take the hyperplane corresponding to $\delta$ to itself. By Theorem \ref{CubeChar}, the corresponding mapping class $\phi$ fixes $\delta$. Similarly, the even-dimensional maximal cubes whose hyperplanes correspond to $A(\epsilon)$ and $B(\epsilon)$ must be sent to the pair of even-dimensional maximal cubes whose hyperplanes correspond to $A(\epsilon')$ and $B(\epsilon')$. The remaining cases are similar.

Finally, a pair of partner curves are simultaneously up or down. By Lemma \ref{3cubes}, they form a 3-cube with $\delta$ while a pair of non-partner curves that are simultaneously up or down do not. It follows that $\phi$ sends partner curves to partner curves.}\end{proof}

While we may conclude that the pair of numbers $|A(\epsilon)|$ and $|B(\epsilon)|$ is equal to $|A(\epsilon')|$ and $|B(\epsilon')|$ if $\Gamma(\epsilon)$ and $\Gamma(\epsilon')$ are $\mbox{Mod}^{*}(S)$-equivalent, we are not yet able to prove Proposition \ref{distinguishing orbits}. We turn to finer invariants of $\Gamma(\epsilon)$.

\section{The labeled polygon $P(\epsilon)$ associated to $\Gamma(\epsilon)$}
\label{polygon section}

Towards the proof of Proposition \ref{distinguishing orbits}, we introduce a more detailed invariant. The information inherent to the invariant we produce is easily packaged as a polygon. The essential tool to building this invariant is the ordering induced on intersection points of a realization of an oriented curve. 

However, while an oriented curve in a realization of a curve system determines an ordering of its intersection points, this ordering may not be an invariant of the collection of homotopy classes; the presence of 3-cubes implies the existence of a Reidemeister move that will make it possible to switch the ordering of intersection points. We state the following only for complete 1-systems for ease in exposition, but with slightly more detail a more general statement could be made. 

\begin{lemma}
\label{ordering sans 3cubes}
{Let $\{\gamma_1,\ldots,\gamma_k,\gamma,\gamma'\}$ be a complete 1-system of curves, and let $\vv{\gamma}$ be a choice of orientation of $\gamma$. Suppose  that $\gamma$ does not form 3-cubes with any pair $\gamma_i$ and $\gamma_j$. Then the cyclic ordering of $\gamma_1,\ldots,\gamma_k$ induced by $\vv{\gamma}$ is invariant of the choice of realization. In this setting, the choice of $\gamma_j$ induces a well-defined ordering of $\{\gamma_i : i\ne j\}$. Moreover, if $\gamma$ and $\gamma'$ form 3-cubes with $\gamma_i$, for all $i$, then the cyclic orderings of $\gamma_1,\ldots,\gamma_k$ induced by the two orientations of $\gamma'$ coincide with those of the two orientations of $\gamma$.}
\end{lemma}

\begin{proof}
{Choose a realization of the curve system $\{\gamma_1,\ldots,\gamma_k,\gamma,\gamma'\}$. One obtains a cyclic ordering of $\{\gamma_1,\ldots,\gamma_k\}$ induced by $\vv{\gamma}$. This ordering is invariant of the chosen realization, since any realization can be obtained from any other realization by applying a sequence of Reidemeister moves, and births or deaths of monogons or bigons \cite[Lemma 5.6]{goldman2}. 

Finally, if $\gamma$ and $\gamma'$ form 3-cubes with each of $\gamma_i$ and $\gamma_j$, but the orderings determined by the orientations of $\gamma$ and $\gamma'$ differ on the pair $\{\gamma_i,\gamma_j\}$, then it follows that there would be a triangle formed by $\gamma$, $\gamma_i$, and $\gamma_j$.}
\end{proof}

We refer to maximal sets of non-partner up (resp.~down) curves as \emph{full}, and we fix choices of full sets of up and down curves $U$ and $D$, respectively. In our setting, by Lemma \ref{3cubes} the curve $\alpha_i$ does not form 3-cubes with any pair of non-partner down curves. Thus, choosing an orientation $\vv{\alpha_i}$, we may apply Lemma \ref{ordering sans 3cubes} to $\vv{\alpha_i}$ with $D$. We conclude that a choice of orientation $\vv{\alpha_i}$ induces a cyclic order to any full set of down curves.

Moreover, as long as $|A(\epsilon)|>2$, the cyclic order of $D$ induced by $\vv{\alpha_i}$ can be upgraded, canonically, into a bona fide ordering: If $|A(\epsilon)|>2$, there is an up curve $\alpha_j$ which is not a partner of $\alpha_i$. By Lemma \ref{3cubes}, the curve $\alpha_i$ does not form a 3-cube with $\alpha_j$ and $\beta_l$, for any $l$, so that Lemma \ref{ordering sans 3cubes} applies to $\vv{\alpha_i}$ and the union of $D$ with $\alpha_j$. Moreover, it is evident that this order does not depend on the choice of $\alpha_j$ among up curves that are not partners of $\alpha_i$. We conclude that a choice of orientation $\vv{\alpha_i}$ induces an ordering of $D$, which we refer to as the $\vv{\alpha_i}$-ordering of $D$.

We choose an almost realization of $\Gamma(\epsilon)$ such that the curves in $U$ intersect at a single point $p$, and curves in $D$ intersect at a single point $p'$. Let $c$ be a small circle centered at $p$, and $c'$ a small circle centered at $p'$. An orientation of $c$ (resp. $c'$) induces a cyclic ordering on the finite set of points $U \cap c$ (resp. $D \cap c')$. Assuming that both $U$ and $D$ are non-empty, using the orientation of the surface $S$, we equip $c$ with a clockwise orientation, and $c'$ with a counter-clockwise orientation. This induces a cyclic ordering on the finite collection of points $U \cap c$ and on $D \cap c'$.

\begin{remark} We note the cyclic ordering on $U \cap c$ (and on $D \cap c'$) described above is an invariant of the $\mbox{Mod}^*(S)$-orbit of $\Gamma(\epsilon)$. This follows from the fact that it corresponds to the cyclic order on the set of endpoints of lifts to $\mathbb{H}^{2}$ of geodesics in $U$ on some hyperbolic surface, induced from an orientation of $\partial \mathbb{H}^{2}$, and therefore this ordering is detected by the dual cube complex to $\Gamma(\epsilon)$. Thus, the definition of this cyclic ordering on $U \cap c$ does not depend on our particular choice of almost realization for $\Gamma(\epsilon)$. 
\end{remark}

Given a choice of orientation of a curve $\gamma\in U$ (resp.~$D$), the two points of $\gamma\cap c$ (resp.~$\gamma\cap c'$) are each equipped with an arrow that either points away from, or towards $p$ (resp.~$p'$). We refer to arrows pointing towards $p$ (resp.~$p'$) as \emph{inward pointing} and the others as \emph{outward pointing}. Given a set of choices of orientations for each of the curves in $U$, the set of inward and outward pointing arrows partitions the set of points $U\cap c$ into two sets. Moreover, there is an involution $\iota$ of $U \cap c$ that exchanges these two sets, for any set of choices of orientations for the curves in $U$: Given $v\in U \cap c$, $\iota(v)$ is the only other point of $U \cap c$ on the same curve of $\Gamma(\epsilon)$ as $v$.

\begin{lemma} \label{Cohere} There exists a choice of orientations $\{\vv{\alpha_i}\}$ for the curves in $U$ such that the $\vv{\alpha_i}$-ordering and the $\vv{\alpha_j}$-ordering on $D$ are cyclically equivalent. 
\end{lemma} 
\begin{proof} Choose orientations for $\alpha_i$ and $\alpha_j$ as in Figure \ref{orderingPic}. The down curves $D$ are partitioned by this choice into four sets $D_1,D_2,D_3,D_4$, and with the chosen orientations on $\alpha_i$ and $\alpha_j$, the $\vv{\alpha_i}$- and the $\vv{\alpha_j}$-orderings of $D$ are cyclically equivalent.
\end{proof}

 \begin{figure}[]
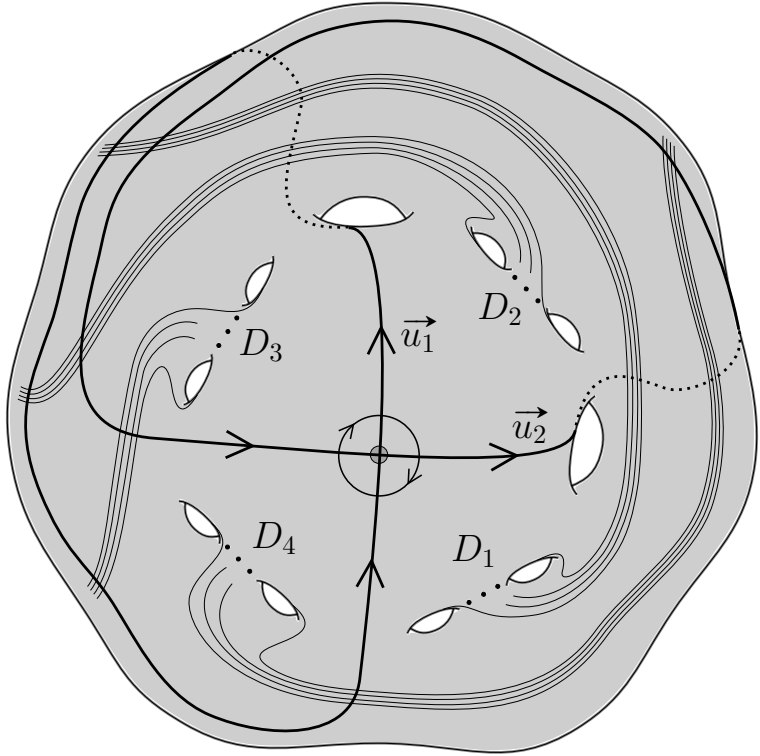

 	\centering
	\Large
 	\begin{lpic}{prop75drawing(10cm)}
		\lbl[]{126,128;$\vv{u_1}$}
		\lbl[]{160,100;$\vv{u_2}$}
		\lbl[]{151,136;$D_2$}
		\lbl[]{78,125;$D_3$}
		\lbl[]{82,66;$D_4$}
		\lbl[]{143,62;$D_1$}
	\end{lpic}
	\caption{The oriented up curve $u_1$ induces the ordering $(D_1,D_2,D_3,D_4)$ of the down curves, while $u_2$ induces $(D_4,D_1,D_2,D_3)$.}
	\label{orderingPic}
\end{figure}

A set of choices of orientations of the curves in $U$ is \textit{coherent} if it satisfies the conclusion of Lemma \ref{Cohere}. The orientations chosen in Lemma \ref{Cohere} (see Figure \ref{orderingPic}) are a convenient choice, and we refer to these choices of orientations of up curves as the \emph{standard orientations} for curves in $U$. It will be useful to have a chosen orientation of the down curves as well. Note the orientation-reversing involution of the surface that exchanges each up curve $\alpha_i$ with the down curve $\beta_i$. An orientation of a down curve is \emph{standard} if the image under this involution is a standardly oriented up curve. 

\begin{remark} \label{CoherentUnique} We note that there are exactly two coherent orientations on $U$; this follows from the fact that a coherent orientation is completely determined by choosing the orientation on one curve of $U$. We will refer to the coherent set of choices of orientations for $U$ that is not the standard one as \emph{non-standard}. 
\end{remark}

\begin{definition} The \textit{labeled polygon} $P(\epsilon)$ associated to $\Gamma(\epsilon)$ is a $2|U|$-gon, satisfying:

\begin{enumerate}
\item The vertices are labeled from elements of $U \cap c$, in the cyclic ordering induced by the orientation of $c$; thus $P(\epsilon)$ comes equipped with a preferred orientation of its boundary. By placing an outward pointing arrow at each point of $U\cap c$, the vertices of $P(\epsilon)$ determine orderings of $D$. By Remark \ref{CoherentUnique}, these orderings partition the vertices into two sets $P_1$ and $P_2$, exchanged by $\iota$, that correspond to the two cyclic equivalence classes of orderings of $D$ determined by possible coherent choices of orientations for $U$. For each $l=1,2$, we let $R_l(\epsilon)$ denote the polygon formed by the cyclically ordered vertices in $P_l$. Note that the edges of $R_l(\epsilon)$ are diagonals of $P(\epsilon)$;

\item Each edge of $R_l(\epsilon)$ is decorated with a pair of integers $\mathcal{M}_l(e)$and $\mathcal{N}_l(e)$, defined as follows: 
Let $e= (\vv{\alpha_{i}}, \vv{\alpha_{j}})$ be an edge of $R_l(\epsilon)$ with initial and terminal vertices corresponding to the oriented curves $\vv{\alpha_{i}}$ and $\vv{\alpha_{j}}$ in $U$, respectively. We set $\mathcal{M}_l(e)$ equal to the number of vertices of $P(\epsilon)$ between $\vv{\alpha_{i}}$ and $\vv{\alpha_{j}}$. If $\beta_{l_{1}},\ldots,\beta_{l_{m}}$ denotes the $\vv{\alpha_{i}}$-ordering of $D$, then by construction there exists some $r$ such that the $\vv{\alpha_{j}}$-ordering on $D$ is $$\beta_{l_{m-r+1}}, \ldots, \beta_{l_m},\beta_{l_{1}},\ldots,\beta_{l_{m-r}}.$$ 
We define $\mathcal{N}_l(e)= r$. 
\end{enumerate}
\end{definition}

\begin{lemma} \label{WellDef} $P(\epsilon)$ is a well-defined invariant of $\Gamma(\epsilon)$. 
\end{lemma}

\begin{proof} By Remark \ref{CoherentUnique}, the standard and non-standard pair of choices of orientations for the curves in $U$ are the only coherent such choices. Thus the partition of the vertices of $P(\epsilon)$ into the two sets $P_1$ and $P_2$ is well-defined. The labels $\mathcal{M}_l(e)$ and $\mathcal{N}_l(e)$ are evidently invariant under a different set of choices of full sets $U$ and $D$.
\end{proof}

\begin{lemma} \label{sumsN} For each of the labeled polygons $P(\epsilon)$ we have 
\begin{align*}
\sum_e \mathcal{N}_{l_1}(e) & = |D|, \text{ and} \\
\sum_e \mathcal{N}_{l_2}(e) & = |D| \cdot (|U| - 1),
\end{align*}
where $\{l_1,l_2\} = \{1,2\}$.
\end{lemma}

\begin{proof}
One of the polygons has vertices that determine the standard orientations of the curves in $U$, while the other has vertices that determine the non-standard ones. By examining Figure \ref{orderingPic}, one sees that the sum for the polygon with standard orientations determined at its vertices is the first sum above, while the sum for the other polygon is the second sum above.
\end{proof}

After relabeling, we assume from now on that $R_1(\epsilon)$ refers to the polygon whose vertices are points of $U\cap c$ with outward pointing arrows in the standard orientations of the curves in $U$. See Figure \ref{labeledPolygons} for examples corresponding to $\epsilon=(1,1,1-1,-1,1,-1)$ and $\epsilon=(-1,1,1,-1,-1,1,1)$.

\begin{figure}[h]
	\begin{minipage}[]{\linewidth}
	\centering
	\begin{lpic}{labelledPolygon1(13cm)}
		\lbl[]{296.5,150;$(1,1)$}
		\lbl[]{367.5,158.5;$(0,1)$}
		\lbl[]{397.5,121;$(1,0)$}
		\lbl[]{347,78;$(2,1)$}
	\end{lpic}
	\caption*{$\epsilon=(1,1,1-1,-1,1,-1)$}
	\end{minipage}
	\vspace{.3cm}\\
	\begin{minipage}[]{\linewidth}
	\centering
	\begin{lpic}{labelledPolygon2(13cm)}
		\lbl[]{292,125;$(1,0)$}
		\lbl[]{340,168;$(0,2)$}
		\lbl[]{395,121;$(1,0)$}
		\lbl[]{340,74;$(2,1)$}
	\end{lpic}
	\caption*{$\epsilon=(-1,1,1,-1,-1,1,1)$}
	\end{minipage}
\caption{Full sets of curves from the curve system $\Gamma(\epsilon)$, the `small circle' $c$, the polygon $P(\epsilon)$, and the polygon $R_1(\epsilon)$. Edge labels are written as the ordered pair $(\mathcal{M}_1(e),\mathcal{N}_1(e))$.}
\label{labeledPolygons}
\end{figure}

\begin{definition} An \textit{isomorphism of labeled polygons} from $P(\epsilon)$ to $P(\epsilon')$ is a permutation $\psi \in S_{2|U|}$, where $S_{n}$ denotes the symmetric group on $n$ symbols, that induces a label-preserving isomorphism between the labeled $1$-skeletons of $P(\epsilon) \cup R_1(\epsilon) \cup R_2(\epsilon)$ and $P(\epsilon')\cup R_1(\epsilon') \cup R_2(\epsilon')$. 
\end{definition}

By Proposition \ref{MCG-orbit restrictions}, a homeomorphism of the surface that takes $\Gamma(\epsilon)$ to $\Gamma(\epsilon')$ takes $A(\epsilon)$ to either $A(\epsilon')$ or $B(\epsilon')$. In the first case, by construction, the homeomorphism induces an isomorphism of labeled polygons $P(\epsilon)\cong P(\epsilon')$. Note that it is a consequence of Lemma \ref{sumsN} that we may conclude that, if $\psi: P(\epsilon)\cong P(\epsilon')$ is an isomorphism of labeled polygons, then $\psi$ takes $R_1(\epsilon)$ to $R_1(\epsilon')$ and $R_2(\epsilon)$ to $R_2(\epsilon')$.

\begin{remark}
\label{preservesR1}
As a consequence, any homeomorphism taking $\Gamma(\epsilon)$ to $\Gamma(\epsilon')$ that takes $A(\epsilon)$ to $A(\epsilon')$ must preserve the standard orientations of curves.
\end{remark}

\begin{proposition} \label{Iso} If $\psi: P(\epsilon) \cong P(\epsilon')$ is an isomorphism of labeled polygons, then $\epsilon$ and $\epsilon'$ are in the same $(\mathbb{Z}/2\mathbb{Z} \oplus D_{g})$-orbit of $\left\{-1,1\right\}^{g}$. 
\end{proposition}

\begin{proof}  
Our strategy below is to make various choices for labels and orientations of curves that make the structure of $P(\epsilon)$ more transparent--by Lemma \ref{WellDef} these choices are allowed. With the structure of $P(\epsilon)$ clear, the result will follow easily. Roughly speaking, the constructions below are careful treatments of `how the curves in $\Gamma(\epsilon)$ look,' when they are drawn conveniently on the surface. In what follows, we fix the almost realization described at the beginning of this section, and fix as well the standard orientations of each of the curves.

We first establish a cyclic ordering on all the curves in $U \cup D$. Consider the collection of curves $\omega_{1},\ldots,\omega_{g}$ in Figure \ref{transversals}, realized in minimal position with $\Gamma(\epsilon)$. Note that for each $i$, the curve $\omega_{i}$ intersects exactly one curve in $U \cup D$. Given $\gamma \in U \cup D$, the \textit{transversal} to $\gamma$ is the unique curve in $\left\{\omega_{1},\ldots, \omega_{g} \right\}$ intersecting $\gamma$. Thus there is a cyclic ordering on $U \cup D$ associated to the cyclic ordering $(1,\ldots,g)$ on the indices of the $\omega_{i}$ curves. 

\begin{figure}[]
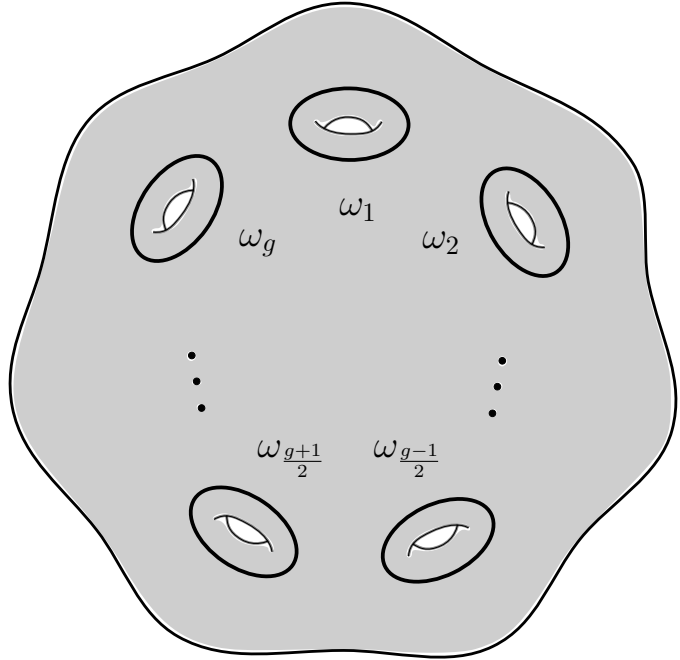

	\centering
	\Large
	\begin{lpic}{OmegaCurves(3.5in)}
		\lbl[]{105,135;$\omega_1$}
		\lbl[]{130,125;$\omega_2$}
		\lbl[]{120,60;$\omega_{\frac{g-1}{2}}$}
		\lbl[]{85,60;$\omega_{\frac{g+1}{2}}$}
		\lbl[]{75,125;$\omega_g$}
	\end{lpic}
	\caption{The transversals $\omega_1,\ldots,\omega_g$.}
	\label{transversals}
\end{figure}

We will now describe a cyclic order on $(U \cap c) \sqcup (D \cap c')$. For $\gamma \in U \cup D$, suppose $\gamma \in U$, and let $\omega$ denote the transversal to $\gamma$. Then $c$ separates $\gamma$ into two sub-arcs, both bounded by the two points $s_{1}(\gamma), s_{2}(\gamma)$ of $\gamma \cap c \subset U \cap c$. Let $\Lambda(\gamma)$ denote the sub-arc whose interior is not contained in the disk bounded by $c$. Then $\Lambda(\gamma)$ is subdivided further into two sub-arcs, which we denote $S_{1}(\gamma)$ and $S_{2}(\gamma)$; for each $i=1,2$, the arc $S_{i}(\gamma)$ is the sub-arc of $\Lambda(\gamma)$ bounded by $s_{i}(\gamma)$ and by $\gamma \cap \omega$ (see Figure \ref{S1S2arcs}). We remark that the arcs $S_1(\gamma)$ and $S_2(\gamma)$ are distinguished by the choice that $S_{1}(\gamma)$ contains none of the intersection points with $D$ in the chosen realization of $\Gamma(\epsilon)$. Similarly, if $\gamma \in D$ then $\Lambda(\gamma)$ is the sub-arc of $\gamma$ not contained within $c'$, and $\Lambda(\gamma)$ is subdivided into $S_1(\gamma)$ and $S_2(\gamma)$, where $S_{1}(\gamma)$ is the sub-arc of $\Lambda(\gamma)$ containing no intersections with $U$.

\begin{figure}[]
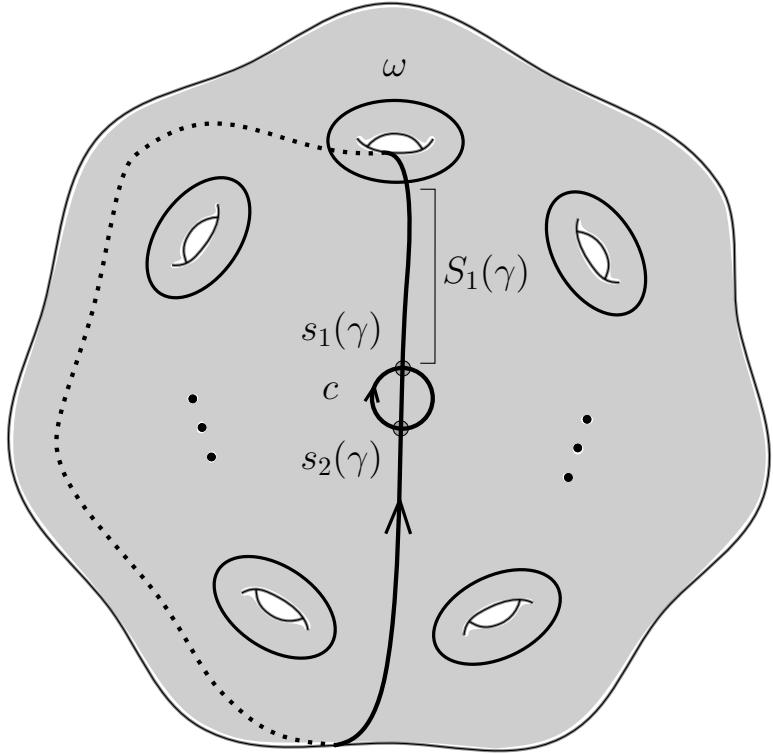

	\centering
	\Large
	\begin{lpic}{S1S2(4.in)}
		\lbl[]{126,126;$S_1(\gamma)$}
		\lbl[]{88,77;$s_2(\gamma)$}
		\lbl[]{88,111;$s_1(\gamma)$}
		\lbl[]{85,95;$c$}
		\lbl[]{102,180;$\omega$}
	\end{lpic}
	\caption{The intersections $s_1(\gamma)$ and $s_2(\gamma)$ of $\gamma \in U$ with the small circle $c$, and the sub-arc $S_1(\gamma)$ of $\gamma$. (The arc $S_2(\gamma)$ is not pictured).}
	\label{S1S2arcs}
	\vspace{.2cm}
\end{figure}

\begin{figure}[]
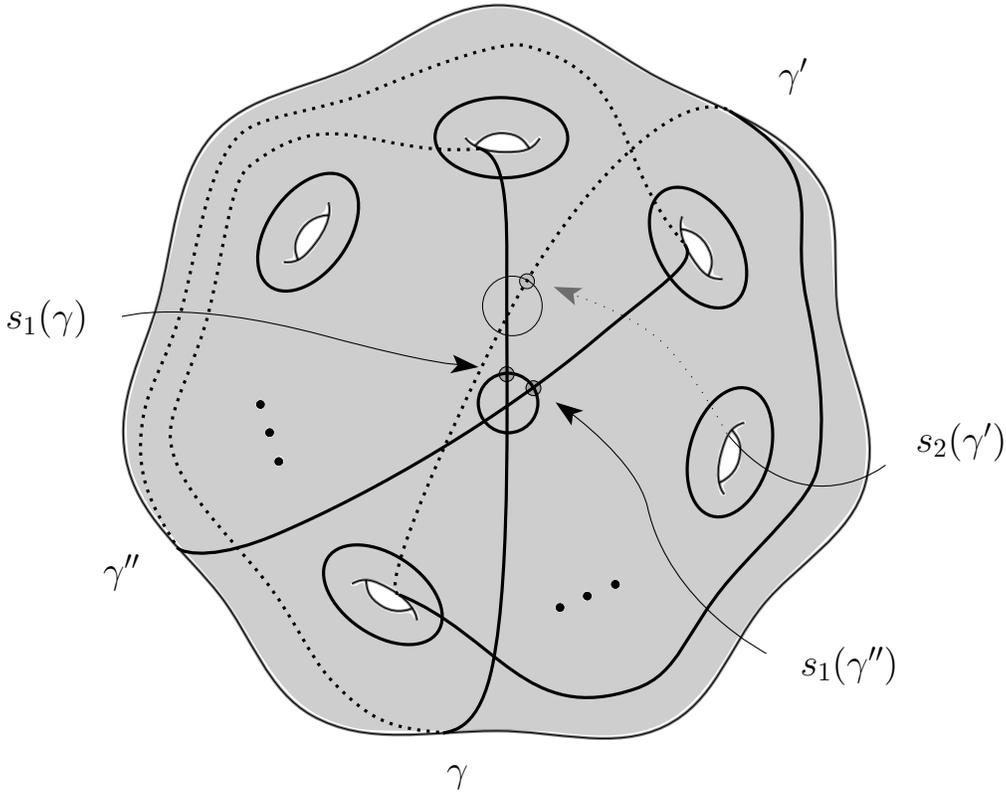

	\centering
	\Large
	\begin{lpic}{cycUD(4in)}
		\lbl[]{-20,113;$s_1(\gamma)$}
		\lbl[]{195,20;$s_1(\gamma'')$}
		\lbl[]{225,80;$s_2(\gamma')$}
		\lbl[]{0,45;$\gamma''$}
		\lbl[]{90,-10;$\gamma$}
		\lbl[]{180,178;$\gamma'$}
	\end{lpic}
	\vspace{.4cm}
	\caption{In the cyclic order on $(U \cap c) \sqcup (D \cap c')$, $s_{2}(\gamma')$ immediately follows $s_{1}(\gamma)$, and $s_{1}(\gamma'')$ immediately follows $s_{2}(\gamma')$.}
	\label{cyclic order}
\end{figure}

Starting at some point $v$ of $U \cap c$, let $\gamma \in U \cup D$ denote the up curve associated to $v$, and let $\omega_{i}$ denote the transversal to $\gamma$. The point of $(U \cap c) \sqcup (D \cap c')$ immediately following $v$ is obtained as follows: We assume that $v= s_{1}(\gamma)$. Let $\gamma'$ be the curve obtained from $\gamma$ by moving $(g+1)/2$ around the cyclic order on $U \cup D$. Thus the transversal to $\gamma'$ is $\omega_{i+ (g+1)/2}$, where the addition is interpreted modulo $g$. Then the point of $(U \cap c) \sqcup (D \cap c')$ immediately following $v$ is $s_{2}(\gamma')$. 

The next point is obtained in a similar fashion: starting at $s_{2}(\gamma')$ we move to the curve whose index in the cyclic order on $U \cup D$ is $(g+1)/2$ from that of $\gamma'$, in the clockwise direction. Letting $\gamma''$ denote this curve, the next point in the cyclic order on $(U \cap c) \sqcup (D \cap c')$ is $s_{1}(\gamma'')$. The ordering is defined by iterating this procedure: add $(g+1)/2$ to the cyclic index, and alternate between $s_{1}$ and $s_{2}$. See Figure \ref{cyclic order}.

Thus far, we have established a cyclic ordering on $U \cup D$, and an associated cyclic ordering on $(U \cap c) \sqcup (D \cap c)$; we also recall that the curves in $U$ and $D$ are equipped with standard orientations. Note that these choices are compatible in the following sense: Recall that the standard orientations of the curves induce inward and outward pointing arrows on the points in $(U\cap c) \sqcup (D\cap c)$. In the established cyclic ordering of $(U\cap c) \sqcup (D\cap c)$, each inward pointing arrow is followed by an outward pointing arrow, and likewise each outward pointing arrow is followed by an inward pointing arrow.

\begin{lemma} \label{Picture} For an edge $e=e(\vv{\alpha_{i}}, \vv{\alpha_{j}})$ of $R_1(\epsilon)$, the integer $\mathcal{M}_1(e)$ (resp.~$\mathcal{N}_1(e)$) is equal to the number of outwardly pointing arrows of $U$ (resp.~$D$) which follow $\iota(\vv{\alpha_{i}})$ and which precede $\iota(\vv{\alpha_{j}})$. 
\end{lemma}

\begin{proof}
With the given choices made, the integer $\mathcal{M}_1(e)$ is equal to the number of inward pointing arrows that follow $\vv{\alpha_i}$ and precede $\vv{\alpha_j}$. Apply the involution $\iota$ and the claim follows for $\mathcal{M}_1(e)$. For $\mathcal{N}_1(e)$, the edge $e$ partitions the down curves into four pieces $D_1,D_2,D_3,D_4$, as in Figure \ref{orderingPic}. The label $\mathcal{N}_1(e)$ of the edge $e$ is equal to $|D_4|$, which is also the number of outwardly pointing arrows between $\iota(\vv{\alpha_{i}})$ and $\iota(\vv{\alpha_{j}})$.
\end{proof}

Henceforth, by an \textit{interval} of $(U \cap c) \sqcup (D \cap c')$, we mean the subset which follows a particular point in $(U \cap c) \sqcup (D \cap c')$ and which precedes some other point, as in the statement of Lemma \ref{Picture}. Evidently, given an interval $I$, $\iota(I)$ is an interval bounded by the image of the end points of $I$ under $\iota$. 

To finish the proof of Proposition \ref{Iso}, it suffices to show that the orbit of $\epsilon$ can be constructed from information about $P(\epsilon)$ which is preserved under isomorphism of labeled polygons. The number of $1$'s in $\epsilon$ is equal to $|A(\epsilon)|/2$, which is equal to half of the number of edges of $P(\epsilon)$. Thus all that remains is to determine how to interleave the $|B(\epsilon)|/2$ necessary $-1$'s to obtain $\epsilon$. The following lemma completes the proof of Proposition \ref{Iso}.
\end{proof}

\begin{lemma} \label{Alt} If $e=(\vv{\alpha_{i}}, \vv{\alpha_{j}})$ is an edge of $R_1(\epsilon)$, then the number of $-1$'s between $\epsilon_i$ and $\epsilon_j$ is equal to $\mathcal{M}_1(e)+\mathcal{N}_1(e)-1$. 

\end{lemma}

\begin{figure}[h]
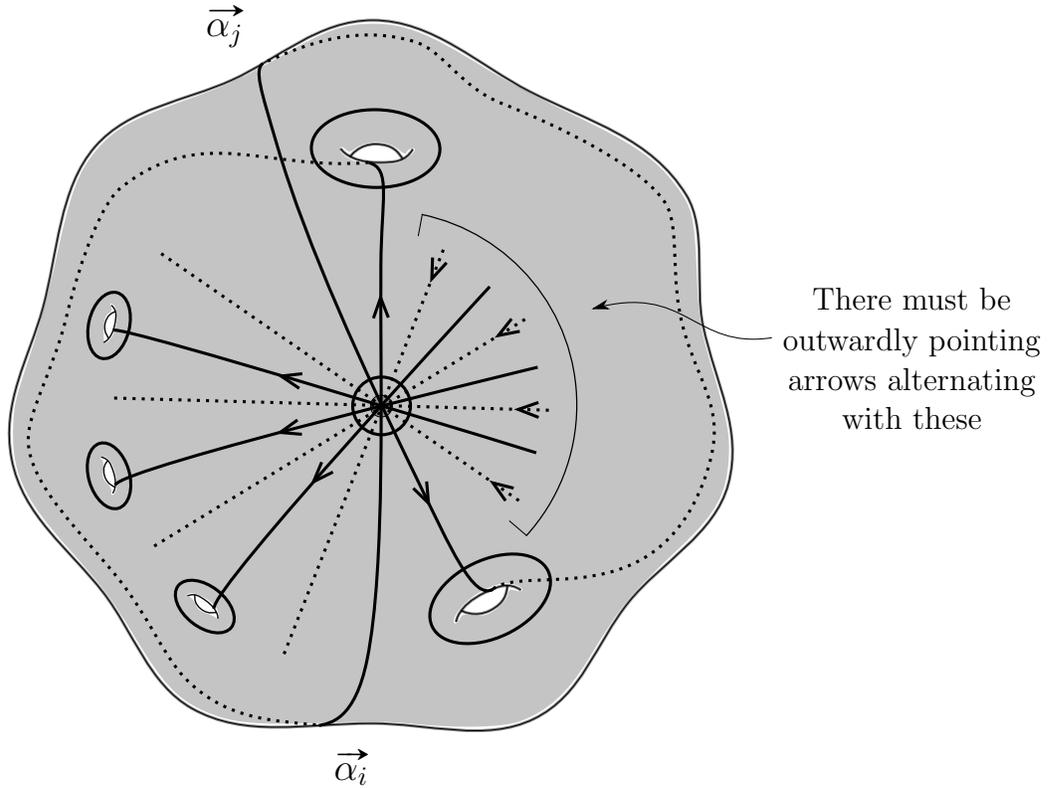

	\centering
	\large
	\begin{lpic}{Alternating(4in)}
		\lbl[]{250,120;There must be}
		\lbl[]{250,108; outwardly pointing}
		\lbl[]{250,97;arrows alternating}
		\lbl[]{250,87;with these}
		\Large
		\lbl[]{95,-10;$\vv{\alpha_i}$}
		\lbl[]{60,195;$\vv{\alpha_j}$}
	\end{lpic}
	\vspace{.5cm}
	\caption{Given standard orientations of the curves in $U$ and $D$, the arrows at the points of $(U\cap c) \sqcup (D\cap c')$ alternate between inward and outward pointing.}
	\label{alternatingPic}
\end{figure}

\begin{proof} Using the standard orientations for the curves in $U\cup D$ again, the consecutive elements of $(U \cap c) \sqcup (D \cap c')$ are equipped with opposite pointing arrows. Each $-1$ between the $1$'s associated to $\vv{\alpha_{i}}$ and $\vv{\alpha_{j}}$ corresponds to a down curve whose outwardly-pointing arrow lies between the outwardly-pointing arrows for $\vv{\alpha_{i}}$ and $\vv{\alpha_{j}}$. Since $e$ is an edge of $R_1(\epsilon)$, there are no outwardly pointing arrows for up curves between the outwardly pointing arrows of $\vv{\alpha_{i}}$ and $\vv{\alpha_{j}}$. 

By construction, the inward and outward pointing arrows alternate in $(U\cap c) \sqcup (D\cap c')$. Thus the number of outward pointing arrows that are between the outward pointing arrows at $\vv{\alpha_{i}}$ and $\vv{\alpha_{j}}$ is equal to one less than the number of inward pointing arrows in this same interval $I$ (see Figure \ref{alternatingPic}). Each inward pointing arrow in this interval corresponds to an outward pointing arrow in $\iota(I)$.

In turn, each outward pointing arrow in $\iota(I)$  is associated to either an up curve or a down curve-- and is therefore accounted for by either $\mathcal{M}_1(e)$ or $\mathcal{N}_1(e)$ by Lemma \ref{Picture}. This completes the proof of Lemma \ref{Alt}. 
\end{proof}

We are now ready to prove Proposition \ref{distinguishing orbits}:

\begin{proof}[Proof of Proposition \ref{distinguishing orbits}]
{Let $\phi\in \mbox{Mod}^*(S)$ have $\phi\cdot \Gamma(\epsilon)= \Gamma(\epsilon')$. By Proposition \ref{MCG-orbit restrictions}, we have either $\phi \cdot A(\epsilon) =A(\epsilon')$ or $\phi \cdot A(\epsilon) = B(\epsilon')$. After composing if necessary with the reflection of the surface that exchanges all up curves $A$ with all down curves $B$ (preserving the $(\Z/2\Z\oplus D_{g})$-orbit of $\epsilon'$), we assume that $\phi\cdot A(\epsilon) = A(\epsilon')$. In this case, $\phi$ induces an isomorphism of the labeled polygons $P(\epsilon) \cong P(\epsilon')$. By Proposition \ref{Iso}, this implies that $\epsilon$ and $\epsilon'$ are in the same $(\Z/2\Z\oplus D_{g})$-orbit.}
\end{proof}

\section{Stabilizing 1-systems}
\label{stabilizing}

We return to stabilization. Suppose $\lambda$ is an almost realization of a maximal complete 1-system on the surface $S_g$ of genus $g$. Given any arc $\alpha$ on $S_g$, we can delete tiny disks at its endpoints (in the complement of any of the curves of $\lambda$), and glue in an annulus. The result is a surface $S_{g+1}$ of genus $g+1$, and a complete 1-system naturally in correspondence with $\Gamma$. We refer to this system of curves by $\Gamma$ as well, the distinction being clear from context. Note that $\Gamma$ is not maximum in $S_{g+1}$, as it consists of two too few curves. 

When $\alpha$ intersects each of the curves of $\lambda$ once (on $S_g$), there are two readily available curves to add: Concatenate $\alpha$ with an arc that crosses the annulus to form a new simple closed curve $\alpha'$. By construction, $\alpha'$ intersects each of the curves of $\Gamma$ once. Moreover, the Dehn twist of $\alpha'$ around the core curve of the annulus, which we denote $\alpha''$, intersects each of the curves of $\Gamma\cup\{\alpha'\}$ once. Thus, given an arc $\alpha$ that intersects each of the curves of a realization of a maximum complete 1-system $\lambda$ once, we produce the \emph{stabilization} of $\lambda$ along the arc $\alpha$, denoted $\mathrm{stab}(\lambda,\alpha):=\Gamma\cup\{\alpha',\alpha''\}$, a maximum complete 1-system on $S_{g+1}$. The following is immediate:

\begin{lemma}
\label{stabilized 3-cubes general 1}
{For each $\gamma\in\Gamma$, the trio $\{\alpha',\alpha'',\gamma\}$ in $\mathrm{stab}(\lambda,\alpha)$ forms a triangle. If three curves don't form a triangle in $\Gamma$, then the corresponding curves don't form a triangle in $\mathrm{stab}(\lambda,\alpha)$.}
\end{lemma}

There is a special case of this stabilization procedure: Choose a curve $\gamma\subset\lambda$, and fix an identification of $\gamma$ with $S^1=[0,1]/0\sim1$, and the image of 0 with $p\in\gamma$. Identify an $\epsilon$-neighborhood of $\gamma$ with the annulus $S^1\times(0,1)$, with coordinates chosen, with $\epsilon$ small enough, such that the intersections of the other curves in $\lambda\setminus\gamma$ with the annulus each consist of a single vertical arc. In these coordinates, let $\alpha$ be the arc $$\left\{ \left( t , \frac{1}{2}-t \right) : t\in[0,1] \right\}.$$

We will refer to the stabilization $\mathrm{stab}(\lambda,\alpha)$ by $\mathrm{stab}(\lambda,p,\gamma)$, which we identify with $\Gamma\cup\{\gamma',\gamma''\}$. Note that the stabilization obtained may depend on the realization $\lambda$ and on the point $p\in\gamma$ chosen, in the sense that it is possible that $\mathrm{stab}(\lambda,p,\gamma)$ and $\mathrm{stab}(\eta,q,\gamma)$ are $\mbox{Mod}^*(S_{g+1})$-inequivalent, either for $p\ne q$ or for non-isotopic realizations $\lambda \not\simeq \eta$.

In the case of this special version of stabilization, we note that there is a one-holed torus subsurface $\Sigma\subset S_{g+1}$ so that the curves in $\Gamma \setminus \{ \gamma \}$ are disjoint from $\Sigma$, and so that the curves $\gamma$, $\gamma'$, and $\gamma''$ are homotopic as properly embedded arcs in the surface with boundary $S_{g+1} \setminus \Sigma$.

\begin{lemma}
\label{stabilized 3-cubes general 2}
{For each $\beta \in \Gamma \setminus \{\gamma\}$, the trios $\{\beta,\gamma,\gamma'\},\{\beta,\gamma,\gamma''\}\subset \mathrm{stab}(\lambda,p,\gamma)$ form triangles in $S_{g+1}$.}
\end{lemma}

Note that Lemmas \ref{tool1} and \ref{stabilized 3-cubes general 1} imply that any maximum complete 1-system obtained via stabilization has dual cube complex of dimension at least three, and Lemma \ref{stabilized 3-cubes general 2}, with Theorem \ref{3-to-n} then implies that any obtained via the special case above has dimension at least four.

\begin{remark} All of the maximum complete 1-systems in our construction can be obtained by a sequence of the more specialized stabilizing procedure, applied successively to the canonical example in genus 2. In particular, they all have dual cube complexes of dimension at least three. Question \ref{Stab} in \S\ref{intro} presents itself.
\end{remark}

We fix choices of almost realizations for $\Gamma(\epsilon)$ as in \S\ref{polygon section}, which we denote as well by $\Gamma(\epsilon)$, the distinction being clear from context.

\begin{lemma}
\label{stabilized 3-cubes}
{For any choice of $p\in\delta$, any trio from $A(\epsilon)\subset\mathrm{stab}(\Gamma(\epsilon),p,\delta)$ (resp.~$B(\epsilon)\subset\mathrm{stab}(\Gamma(\epsilon),p,\delta)$) forms a triangle in $S_{g+1}$.}
\end{lemma}

\begin{proof}
{In the chosen realizations, the triangles formed among the curves of $A(\epsilon)$ are all disjoint from the stabilizing arcs (which track $\delta$ very closely), and from the disks chosen in the stabilization process. Thus the triangles persist.}
\end{proof}

\begin{proposition}
\label{MCG-orbit restrictions stabilized}
{If $\phi\in \mbox{Mod}^{*}(S_{g+1})$ satisfies $$\phi\cdot\mathrm{stab}(\Gamma(\epsilon),p,\delta)=\mathrm{stab}(\Gamma(\epsilon'),p,\delta), \text{ then}$$ 
\begin{enumerate}
\item $\phi\cdot \{\delta,\delta',\delta''\}=\{\delta,\delta',\delta''\}$, and
\item $\phi\cdot \left\{A(\epsilon),B(\epsilon)\right\}=\left\{A(\epsilon'),B(\epsilon')\right\}$.
\end{enumerate}}
\end{proposition}

\begin{proof}
{By Lemma \ref{stabilized 3-cubes}, the hyperplanes corresponding to curves of $A(\epsilon)$ (resp.~$B(\epsilon)$) pass through maximal cubes of dimension $|A(\epsilon)|$ (resp.~$|B(\epsilon)|$), which is even. The hyperplanes corresponding to $\delta$, $\delta'$, and $\delta''$ pass through maximal cubes of odd dimension (e.g.~the cube corresponding to the curves $\delta$, $\delta'$, $\delta''$, $\alpha_{2i-1}$, $\alpha_{2i}$, $\beta_{2j-1}$, and $\beta_{2j}$, when $|A(\epsilon)|,|B(\epsilon)|\ge1$). The conclusion follows as in the proof of Proposition \ref{MCG-orbit restrictions}.}
\end{proof}

We are now able to prove Proposition \ref{stabilizing Gammas}:

\begin{proof}[Proof of Proposition \ref{stabilizing Gammas}]
{Suppose $\phi\in \mbox{Mod}^*(S_{g+1})$ sends $\mathrm{stab}(\Gamma(\epsilon),p,\delta)$ to $\mathrm{stab}(\Gamma(\epsilon'),p,\delta)$, so that by Proposition \ref{MCG-orbit restrictions stabilized} we have that $\phi$ preserves the set $\{\delta,\delta',\delta''\}$.

Consider the one-holed torus subsurface $\Sigma \subset S_{g+1}$ that arises in the course of the stabilizations $\mathrm{stab}(\Gamma(\epsilon),p,\delta)$. Consider the mapping class $\tilde{\phi}$ induced on $S_g$ by deleting $\Sigma$ and identifying the resulting boundary component to a point. Since the curves of $\Gamma(\epsilon)\setminus \{\delta\}$ (and likewise $\Gamma(\epsilon')\setminus\{\delta\}$) can be made disjoint from $\Sigma$, the induced map $\tilde{\phi}$ takes the induced curves $\Gamma(\epsilon)\setminus \{\delta\}$ to $\Gamma(\epsilon')\setminus\{\delta\}$. Moreover, in the resulting surface $\delta,\delta',$ and $\delta''$ go to the same homotopy class of curve, namely $\delta \subset S_g$. Since $\phi\cdot \{\delta,\delta',\delta''\}=\{\delta,\delta',\delta''\}$, we find that $\tilde{\phi}\cdot \delta = \delta$, so that $\Gamma(\epsilon)$ and $\Gamma(\epsilon')$ are equivalent under the action of $\mbox{Mod}^*(S_g)$.}
\end{proof}

\end{document}